\documentclass{amsart}
\usepackage{bbm}
\usepackage{amsfonts}
\usepackage{amsmath,amsthm,amscd,mathrsfs,amssymb,graphicx,enumerate,
  dsfont}
\usepackage{xypic}

\newtheorem{thm}{Theorem}[section]
\newtheorem{cor}[thm]{Corollary}
\newtheorem{lem}[thm]{Lemma}
\newtheorem{prop}[thm]{Proposition}
\newtheorem{conj}[thm]{Conjecture}
\theoremstyle{definition}
\newtheorem{defn}[thm]{Definition}
\newtheorem{assm}[thm]{Assumption}
\newtheorem{rmk}[thm]{Remark}
\newtheorem{ques}[thm]{Question}

\DeclareMathOperator{\Hom}{Hom} 
 \DeclareMathOperator{\rk}{rk}
 \DeclareMathOperator{\Sym}{Sym}
\DeclareMathOperator{\Pic}{Pic} 
\DeclareMathOperator{\Span}{Span}

\newcommand{\C}{\ensuremath\mathds{C}}

\newcommand{\Z}{\ensuremath\mathds{Z}}
\newcommand{\Q}{\ensuremath\mathds{Q}}

\newcommand{\fa}{\ensuremath\mathfrak{a}}
\newcommand{\fb}{\ensuremath\mathfrak{b}}

\newcommand{\PP}{\ensuremath\mathds{P}}
\newcommand{\calO}{\ensuremath\mathcal{O}}

\newcommand{\calF}{\ensuremath\mathscr{F}}

\newcommand{\HH}{\ensuremath\mathrm{H}}
\newcommand{\CH}{\ensuremath\mathrm{CH}}

\begin{document}
\title{Hyperk\"ahler manifolds of Jacobian type}
\author{Mingmin Shen}
\thanks{The author is supported by Simons Foundation as a Simons Postdoctoral Fellow.}
\address{DPMMS, University of Cambridge, Wilberforce Road, Cambridge CB3 0WB, UK}
\email{M.Shen@dpmms.cam.ac.uk}


\date{\today}

\begin{abstract}
In this paper we define the notion of a hyperk\"ahler manifold
(potentially) of Jacobian type. If we view hyperk\"ahler manifolds
as ``abelian varieties", then those of Jacobian type should be
viewed as ``Jacobian varieties". Under a minor assumption on the
polarization, we show that a very general polarized hyperk\"ahler
fourfold $F$ of $K3^{[2]}$-type is not of Jacobian type. As a
potential application, we conjecture that if a cubic fourfold is
rational then its variety of lines is of Jacobian type. Under some
technical assumption, it is proved that the variety of lines on a
rational cubic fourfold is potentially of Jacobian type. We also
prove the Hodge conjecture in degree 4 for a generic $F$ of $K3^{[2]}$-type.
\end{abstract}

\maketitle

\section{Introduction}
The study of weight one Hodge structures leads to the theory of
abelian varieties. Among all principally polarized abelian
varieties, there is a special class called Jacobian varieties which
corresponds to Hodge structures of curves. A corresponding theory
for higher weight Hodge structures is still missing. In this paper,
we propose a possible direction of a theory of ``Jacobians" for
weight two Hodge structures.

We do not know yet what the ``principally polarized abelian
varieties" for weight two Hodge structures should be. The most
typical weight two Hodge structure is $\HH^2(S,\Z)$ where $S$ is a
smooth projective surface, which is equipped with the intersection blinear
form. A ``p.p.a.v." in weight two should be a variety $X$
whose cohomology group $\HH^2(X,\Z)$ carries a natural nondegenerate
bilinear form that is compatible with the Hodge structure. We know
one class of such varieties, namely the hyperk\"ahler manifolds $F$
whose cohomology group $\HH^2(F,\Z)$ carries the canonical
Beauville--Bogomolov bilinear form, see \cite{beauville},
\cite{huybrechts}. Examples of such manifolds include generalized Kummer varieties, Hilbert scheme of points
on a $K3$ surface and the variety of lines on a smooth cubic fourfold.

To motivate the definition of ``Jacobians" in weight two, we recall
a characterization of Jacobian of curves. Let $(A,\Theta)$ be a
principally polarized abelian variety of dimension $g$. Then a
result of Matsusaka \cite{matsusaka} says that $A$ is a Jacobian if
and only if there exists a curve $f:C\to A$ such that
$f_*[C]=\frac{[\Theta]^{g-1}}{(g-1)!}$ is the 1-dimensional minimal
cohomology class. Or in other words, $A$ is a Jacobian if and only
if the 1-dimensional minimal cohomology class
$\frac{[\Theta]^{g-1}}{(g-1)!}$ is effective. We define the
``Jacobian" in weight two as follows.
\begin{defn}
Let $F$ be a hyperk\"ahler manifold of dimension $2n$. Let $b(-,-)$
be the Beauville--Bogomolov bilinear form on $\HH^2(F,\Z)$. Let
$$
\mathrm{Alg}^{2p}(F)=\mathrm{Im}(cl:\mathrm{CH}^p(F)\to\HH^{2p}(F,\Z)),\quad
\mathrm{Hdg}^{2p}(F)=\HH^{2p}(F,\Z)\cap \HH^{p,p}(F)
$$
be the subgroups of algebraic and Hodge classes. We define the
\textit{transcendental lattice} of $F$ to be
$$
\HH^2(F,\Z)_{\mathrm{tr}}=\mathrm{Alg}^2(F)^\perp
$$
with the restriction of the bilinear form $b$. A \textit{minimal
class} is an element $\theta\in\HH^{4n-4}(F,\Z)$ such that
$$
(\theta\cdot\alpha_1\cdot\alpha_2)_F=b(\alpha_1,\alpha_2),\quad\forall\alpha_1,\alpha_2\in\HH^2(F,\Z)_\mathrm{tr}
$$
A \textit{minimal Hodge class} is a minimal class that is also a
Hodge class. A projective hyperk\"ahler manifold $F$ is of
\textit{Jacobian type} if there exists a surface $f:S\to F$ such that
$f_*[S]$ is a minimal class.
\end{defn}
By definition, a necessary condition for $F$ to be of Jacobian type
is that it admits a minimal Hodge class. It is easy to see that if
$F=S^{[n]}$ for some $K3$-surface $S$, then $F$ is of Jacobian type.
In fact, we can fix $n-1$ general points on $S$ and let the
$n^\text{th}$ point vary and get a surface $\tilde{S}\subset F$. We
can take $\theta$ to be the cycle class of $\tilde{S}$ and verify
that $F$ is of Jacobian type. From now on, we will restrict
ourselves to the case where $F$ is of $K3^{[2]}$-type, i.e. deformation equivalent to $S^{[2]}$. The following
definitions will be useful throughout this article.

\begin{defn}
Let $F$ be a hyperk\"ahler manifold of $K3^{[2]}$-type.
Let $b(-,-)$ be the Beauville--Bogomolov bilinear form on
$\HH^2(F,\Z)$. An element $\alpha\in\HH^2(F,\Z)$ is
\textit{primitive} if it is not divisible (in $\HH^2(F,\Z)$) by any
integer greater than 1. A primitive element $\alpha\in\HH^2(F,\Z)$
is \textit{even} if $b(\alpha,\alpha')\in 2\Z$, for all
$\alpha'\in\HH^2(F,\Z)$; otherwise, $\alpha$ is called \textit{odd}.
An element $\delta\in\HH^2(F,\Z)$ is \textit{exceptional} if (i)
$b(\delta,\delta)=-2$ and (ii) $\delta$ is even. Let
$\Omega(F)\subset\HH^2(F,\Z)$ be the set of all exceptional
elements. We use $\Omega_0(F)$ to denote $\Omega(F)/\{\pm1\}$. A
\textit{polarization} $\lambda_0\in\Pic(F)$ of $F$ is the class of
an ample divisor. We say that $(F,\lambda_0)$ is \textit{primitively
polarized} if furthermore $\lambda_0$ is primitive.
\end{defn}
\begin{rmk}
An exceptional class $\delta$ is always primitive and satisfies
$b(\delta,\HH^2(F,\Z))=2\Z$. Let $\delta^\perp$ be the orthogonal
complement (with respect to $b(-,-)$) of $\delta$. Then
$(\delta^\perp,b)$ is isomorphic to the $K3$-lattice and
$$
\HH^2(F,\Z)=\Z\delta\oplus\delta^\perp.
$$
We will show that if $\alpha$ is even, then $\alpha^2-\delta^2$ is
divisible by $8$ in $\HH^4(F,\Z)$, for all $\delta\in\Omega(F)$.
\end{rmk}

For a primitively polarized hyperk\"ahler manifold, $(F,\lambda_0)$,
of $K3^{[2]}$-type, we will refer to the following technical
assumption frequently.
\begin{assm}\label{technical assumption}
Either the polarization $\lambda_0$ is odd or it is even and
$\frac{10+b(\lambda_0,\lambda_0)}{8}$ is an even integer.
\end{assm}

The first main result of this paper is that a generic deformation of
$S^{[2]}$ is not of Jacobian type; see Corollary \ref{jacobian picard rank}.

\begin{thm}\label{main theorem on generic nonjacodbian}
Let $(F,\lambda_0)$ be a primitively polarized hyperk\"ahler manifold of
$K3^{[2]}$-type which satisfies Assumption \ref{technical
assumption}. If $F$ has a minimal Hodge class, then $\Pic(F)$ has
rank at least 2. In particular, a very general such $(F,\lambda_0)$
has no minimal Hodge class and hence is not of Jacobian type.
\end{thm}

Our motivation for this study is the rationality problem of cubic
fourfolds. Works concerning this problem include \cite{hassett},
\cite{kulikov} and \cite{kuznetsov}. In their famous paper
\cite{cg}, Clemens and Griffiths proved that a smooth cubic
threefold is not rational by showing that its intermediate Jacobian
is not a Jacobian. The key point here is that if a 3 dimensional
algebraic variety is rational, then its intermediate Jacobian is the
Jacobian of a curve. As an analogue, we make the following

\begin{conj}\label{rationality conjecture}
Let $X\subset\PP^5_\C$ be a smooth cubic fourfold and $F(X)$ its
variety of lines. If $X$ is rational, then $F(X)$ is of Jacobian
type.
\end{conj}

We have two known classes of rational cubic fourfolds, namely the Pfaffian cubic fourfolds obtained by
Beauville--Donagi \cite{bd} and the rational cubic fourfolds containing a plane obtained by Hassett \cite{hassett2}.
In both cases, the above conjecture can be easily verified; see Remark \ref{rmk known cases}. In \cite{hassett},
B. Hassett defines a cubic fourfold to be \textit{special} if
it contains a surface which is not homologous to a complete
intersection.
\begin{prop}\label{prop jacobian implies special}
Let $X$ be a smooth cubic fourfold and $F=F(X)$ its variety of lines with $g_1$ being the Pl\"ucker polarization
on $F$. Then the polarization $\lambda_0=g_1$ satisfies Assumption \ref{technical assumption}. If $F$ is of
Jacobian type, then $X$ is special.
\end{prop}

Then the following proposition follows immediately.
\begin{prop}\label{conjecture implies generic nonrational}
If $X$ is a very general cubic fourfold, then $F(X)$ is not of
Jacobian type. Conjecture \ref{rationality conjecture} implies that a rational cubic
fourfold is special.
\end{prop}

To give further evidence of the Conjecture \ref{rationality conjecture}, we need the
following definition.
\begin{defn}
Let $F$ be a hyperk\"ahler manifold. We say that $F$ is
\textit{potentially of Jacobian type} if there exists a smooth
surface $S$ (not necessarily irreducible) and a correspondence
$\Gamma\in\CH_2(S\times F)$, such that
$$
([\Gamma]^*\alpha\cdot [\Gamma]^*\beta)_S =
b(\alpha,\beta),\quad\forall\alpha,\beta\in\HH^2(F,\Z)_{\mathrm{tr}}.
$$
\end{defn}
\begin{thm}\label{thm on rationality imply potential}
Let $X$ be a smooth cubic fourfold and $F$ its variety of lines. If
there is a birational map $f:\PP^4\dasharrow X$ whose indeterminacy
can be resolved by a simple successive blow-up, then $F$ is
potentially of Jacobian type.
\end{thm}
For the definition of simple successive blow-up, see Definition
\ref{defn of simple blow up}.

The plan of this paper is as follows. For a hyperk\"ahler manifold
$F$ of $K3^{[2]}$-type, our investigation relies on an explicit
basis for $\HH^4(F,\Z)$ obtained in Theorem \ref{explicit
basis II}. The following
canonical torsion group
$$
\mathcal{T}^4(F)=\HH^4(F,\Z)/\Sym^2(\Lambda),\quad\Lambda=\HH^2(F,\Z)
$$
comes into the picture in a subtle way. We give a description of
this group in Theorem \ref{thm on T^4}. In section 4, we first give an explicit description of a certain subgroup $V_{\lambda_0}\subset \mathrm{Hdg}^4(F)$ of integral Hodge classes; see Definition \ref{defn V lambda} for the notation. This is the technical Lemma \ref{lemma on integral classes} which becomes quite useful later on. We show that a minimal Hodge class in $\HH^4(F,\Z)$ is always contained in the linear span of intersections of divisors and the Beauville--Bogomolov form. In particular, if $F$ has Picard rank one then a minimal Hodge class is always in $V_{\lambda_0}$. From this combined with the technical lemma mentioned above, we deduce Theorem \ref{main theorem on generic nonjacodbian}. In section 5, we carry out an infinitesimal deformation calculation and show that on a very general $F$ the Hodge classes in degree $4$ are generated by the polarization and the Beauville--Bogomolov form, i.e. $\mathrm{Hdg}^4(F)=V_{\lambda_0}$; see Theorem \ref{thm on hodge classes}. This yet gives a second proof of the fact that a very general $F$ is not of Jacobian type. One consequence of this calculation is a proof of the Hodge conjecture for very general $F$. If $F$ is the variety of lines on a very general cubic fourfold, we are able to prove the integral Hodge conjecture thanks to the rich geometry that we have in this case; see Theorem \ref{integral hodge cubic fourfolds}. Section 7 gives a proof of Theorem \ref{thm on rationality imply potential}. This section is independent of the previous ones.

Before ending this introduction, we would like to make a few remarks
on some open questions. We see that $S^{[n]}$ is of Jacobian type,
but we do not know whether the converse holds true or not. Namely,
if we know that a hyperk\"ahler manifold $F$ (deformation equivalent
to $S^{[n]}$) is of Jacobian type, can we conclude that $F$ is
birational to $S^{[n]}$, for some $K3$ surface $S$? If $F=F(X)$ is
the variety of lines on a cubic fourfold, then there is always the
surface $S_l$ of lines meeting a given general line $l$ and the natural
inclusion
$$
f:S_l\to F
$$
such that
$$
f^*\alpha\cdot
f^*\beta=2b(\alpha,\beta),\quad\forall\alpha,\beta\in\HH^2(F,\Z)_{\mathrm{tr}}.
$$
In analogy with the theory of abelian varieties, this is saying that
$F(X)$ is a ``Prym variety". In fact we do have a Prym construction
as follows. There is an involution $\sigma$ on $S_l$ and the image
of the restriction
$\HH^2(F,\Z)_{\mathrm{tr}}\to\HH^2(S_l,\Z)_{\mathrm{tr}}$ is the
part on which $\sigma=-1$, see \cite{izadi}. If we replace $l$ by a
rational curve of higher degree, then we get a Prym-Tjurin
construction, see \cite{pt}. We would like to ask the following
questions. For a hyperk\"ahler manifold $F$, what is the smallest
integer $e>0$ such that there is a surface $f:S\to F$ satisfying
$$
f^*\alpha\cdot
f^*\beta=eb(\alpha,\beta),\quad\forall\alpha,\beta\in\HH^2(F,\Z)_{\mathrm{tr}}?
$$
If $S$ is a surface satisfying the above condition, is there a
Prym-Tjurin construction on $S$ to give the restriction of
$\HH^2(F,\Z)_{\mathrm{tr}}$?

\textit{Acknowledgement.} The author would like to thank Claire
Voisin for many helpful email correspondences. He also thanks D.
Huybrechts, B. Totaro and E. Markman for their comments on an
earlier version of this paper and C. Vial for many helpful discussions.

\section{The degree 4 integral cohomology group}
In this section, we fix $F$ to be a hyperk\"ahler manifold which is
deformation equivalent to the Hilbert scheme of length 2 subschemes
of a $K3$ surface. Such an $F$ will be called a hyperk\"ahler manifold
of $K3^{[2]}$-type. We will give an explicit description of
$\HH^4(F,\Z)$. Some general results about the integral cohomology
ring were obtained in \cite{qin-wang}, \cite{li-qin} and
\cite{markman}.

Let $(\Lambda,b)$ be the lattice $\HH^2(F,\Z)$ equipped with the
Beauville--Bogomolov bilinear form. We define the second symmetric
power of $\Lambda$ to be
$$
\Sym^2(\Lambda)=\Lambda\otimes\Lambda/\langle a\otimes b-b\otimes
a\rangle.
$$
We will simply use $\alpha\beta\in\Sym^2(\Lambda)$ to denote the
image of $\alpha\otimes\beta\in \Lambda\otimes\Lambda$. The cup
product
$$
\cup:\HH^2(F,\Z)\otimes\HH^2(F,\Z)\to\HH^4(F,\Z)
$$
naturally factors through $\Sym^2(\Lambda)$ and induces a
homomorphism
$$
\rho:
\Sym^2(\Lambda)\to\HH^4(F,\Z),\quad\alpha\beta\mapsto\alpha\cup\beta.
$$

\begin{lem}\label{lem sym2}
(i) $\rho\otimes\Q:\Sym^2(\HH^2(F,\Q))\to\HH^4(F,\Q)$ is an isomorphism.\\
(ii) $\rho$ is injective and the image has finite index in $\HH^4(F,\Z)$.\\
(iii) The intersection form on $\HH^4(F,\Z)$ restricted to
$\Sym^2\Lambda$ is given by
$$
\rho(\alpha_1\alpha_2)\cdot\rho(\alpha_3\alpha_4)=b(\alpha_1,\alpha_2)b(\alpha_3,
\alpha_4)+b(\alpha_1,\alpha_3)b(\alpha_2,\alpha_4)+b(\alpha_1,\alpha_4)b(\alpha_2,\alpha_3)
$$
for all $\alpha_1,\ldots,\alpha_4\in\Lambda$.
\end{lem}
\begin{proof}
By a result of Verbitsky \cite{verbitsky, bogomolov}, we know that the subalgebra of $\HH^*(F,\Q)$ generated by $\HH^2(F,\Q)$ is isomorphic to its symmetric algebra modulo an ideal supported at degrees of at least $6$. Hence $\Sym^2\HH^2(F,\Q)\subset \HH^4(F,\Q)$. Namely $\rho$ is injective. At the same time, we know that $b_4(F)=276=\dim\Sym^2H^2(F)$; see \cite{bd}. This implies that $\rho\otimes\Q$ is surjective and hence (i) follows. Statement (ii) is a direct consequence of (i) and (iii) is simply the Fujiki relation; see \cite{beauville}.
\end{proof}

\begin{defn}
We define $\mathcal{T}^4(F)$ to be the quotient of $\HH^4(F,\Z)$ by
$\rho(\Sym^2\Lambda)$.
\end{defn}

Note that the group $\mathcal{T}^4(F)$ is always a torsion group of finite order.

\subsection{The special case of $F=S^{[2]}$}
In this subsection, we carry out some explicit
computations in the situation $F=S^{[2]}$ for some $K3$ surface $S$.
Let $\tau: Z\rightarrow S\times S$ be the blow up of $S\times S$
along the diagonal $\Delta_S:S\to S\times S$. Let
$j:\tilde{\Delta}\hookrightarrow Z$ be the exceptional divisor of
the blow-up. There is a natural morphism $\eta:\tilde{\Delta}\to S$
that realizes $\tilde{\Delta}$ as a $\PP^1$-bundle over $S$.
Actually, $\tilde{\Delta}=\PP(T_S)$ is the projectivization of the
tangent bundle of $S$. There is a natural degree 2 finite morphism
$\pi: Z\to S^{[2]}$ that ramifies along the divisor
$\tilde{\Delta}$. Let $G=\{1,\sigma\}\cong\Z/2\Z$ act on $S\times S$
by switching the two factors. Then this action lifts to an action of
$G$ on $Z$ and $S^{[2]}$ is the associated quotient. Let
$\Delta=\pi(\tilde{\Delta})\subset S^{[2]}$. By construction,
$\pi|_{\tilde{\Delta}}:\tilde{\Delta}\to\Delta$ is an isomorphism.
We use $j'$ to denote the inclusion of $\Delta$ into $S^{[2]}$. In
summary, we have the following picture,
$$
\xymatrix{
 S &\tilde{\Delta}\ar[r]^j\ar[d]_{\cong}\ar[l]_{\eta} & Z\ar[r]^{\tau}\ar[d]^{\pi} &S\times S\\
  &\Delta\ar[r]^{j'} &S^{[2]} &
}
$$
The blow up $\tau$ gives a short exact sequence
$$
\xymatrix{
 0\ar[r] &\HH^2(S\times S,\Z)\ar[rr]^{\tau^*}
 &&\HH^2(Z,\Z)\ar[rr]^{\quad\cdot(-f)} &&\Z\ar[r] &0.
}
$$
The element $f\in \HH^6(Z,\Z)$ is the class of a fiber of
$\eta:\tilde\Delta\to S$. This sequence naturally splits since we
have a splitting homomorphism $\Z\to \HH^2(Z,\Z)$, $1\mapsto
[\tilde\Delta]$. Hence we have a canonical isomorphism
$$
\HH^2(Z,\Z)\cong\HH^2(S\times S,\Z)\oplus \Z[\tilde\Delta].
$$
This is compatible with the $G$-action ($G$ acts trivially on
$[\tilde\Delta]$). Note that by K\"unneth formula we have
$$
\HH^2(S\times S,\Z)= p_1^*\HH^2(S,\Z)\oplus
p_2^*\HH^2(S,\Z)\cong\HH^2(S,\Z)\otimes\Z[G].
$$
From now on, we will fix a $\Z$-basis
$$
\{\fa_1,\fa_2,\ldots,\fa_{22}\}
$$
for $\HH^2(S,\Z)$. The following lemma follows from the above
computations
\begin{lem}\label{second cohomology of Z}
(i)
$\HH^2(Z,\Z)^G=\oplus_{i=1}^{22}\Z\,\tau^*(p_1^*\fa_i+p_2^*\fa_i)\oplus\Z[\tilde\Delta]$.\\
(ii) For all odd $i$, we have $\HH^i(G,\HH^2(Z,\Z))=0$.\\
(iii) For all even positive $i$, we have
$\HH^i(G,\HH^2(Z,\Z))=\Z/2\Z$.
\end{lem}
We have similar descriptions of $\HH^4(Z,\Z)$. First we have the
following exact sequence (derived from the Leray spectral sequence
associated to $\tau:Z\to S\times S$),
\begin{equation}\label{weight 4 cohomology}
\xymatrix{
 0\ar[r] &\HH^4(S\times S,\Z)\ar[r] &\HH^4(Z,\Z)\ar[r]^{\phi\quad}
 &\HH^2(S,R^2\eta_*\Z)\ar[r] &0.
}
\end{equation}
Let $\xi=c_1(\calO_{\tilde\Delta}(1))\in\HH^2(\tilde\Delta,\Z)$,
where $\calO_{\tilde\Delta}(1)$ is the relative $\calO(1)$ bundle of
$\tilde\Delta=\PP(T_S)\to S$. Then $\xi$ defines an isomorphism
$\Z\cong R^2\eta_*\Z$. This induces
$$
\HH^2(S,\Z)\cong\HH^2(S,R^2\eta_*\Z).
$$
Let $G$ act trivially on $\HH^2(S,R^2\eta_*\Z)$, then the sequence
\eqref{weight 4 cohomology} respects the $G$-actions.
\begin{lem}\label{splitting lemma}
(i) The homomorphism $\HH^2(S,\Z)\to\HH^4(Z,\Z)$, $\fa_i\mapsto
-j_*\eta^*\fa_i$, splits the sequence \eqref{weight 4 cohomology}
canonically.\\
(ii) As $G$-modules, we have
$$
\HH^4(Z,\Z)=\HH^4(S\times S,\Z)\oplus \HH^2(S,\Z)
$$
where $G$ acts trivially on the factor $\HH^2(S,\Z)$.
\end{lem}
\begin{proof}
The homomorphism $\phi:\HH^4(Z,\Z)\to\HH^2(S,R^2\eta_*\Z)$ factors
as
$$
\xymatrix{\HH^4(Z,\Z)\ar[r]^{j^*}
&\HH^4(\tilde\Delta,\Z)\ar[r]^{\eta_*} &\HH^2(S,\Z)}.
$$
It is easy to check that
\begin{align*}
\eta_*j^*(-j_*\eta^*\fa)
&=-\eta_*(j^*[\tilde\Delta]\cdot\eta^*\fa)\\
&= -\eta_*(-\xi\cdot\eta^*\fa)\\
&=\eta_*(\xi\cdot\eta^*\fa)\\
&=\fa,\qquad \forall \fa\in\HH^2(S,\Z),
\end{align*}
where $\xi\in\HH^2(\tilde\Delta,\Z)$ is the first Chern class of the
relative $\calO(1)$ bundle. This proves (i). The statement (ii)
follows directly from (i).
\end{proof}

\begin{lem}\label{group cohomology weight 4}
(i) The group $\HH^0(G,\HH^4(S\times S,\Z))$ is freely generated
by $p_1^*[pt]+p_2^*[pt]$, $\{p_1^*\fa_i\otimes
p_2^*\fa_i\}_{i=1}^{22}$ and $\{p_1^*\fa_i\otimes
p_2^*\fa_j+p_1^*\fa_j\otimes
p_2^*\fa_i\}_{1\leq i<j\leq 22}$.\\
(ii) For all odd $i$, we have $\HH^i(G,\HH^4(S\times S,\Z))=0$; for
all even $i>0$, we have $\HH^i(G,\HH^4(S\times
S,\Z))\cong(\Z/2\Z)^{22}$.\\
(iii) The group $\HH^0(G,\HH^4(Z,\Z))$ is freely generated by
$e_0=\tau^*(p_1^*[pt]+p_2^*[pt])$, $\{e_i=\tau^*(p_1^*\fa_i\otimes
p_2^*\fa_i)\}_{i=1}^{22}$, $\{e_{ij}=\tau^*(p_1^*\fa_i\otimes
p_2^*\fa_j+p_1^*\fa_j\otimes p_2^*\fa_i)\}_{1\leq i<j\leq 22}$ and
$\{u'_i=j_*\eta^*\fa_i\}_{1\leq i\leq 22}$.\\
(iv) For all odd $i$, we have $\HH^i(G,\HH^4(Z,\Z))=0$; for all even
$i>0$, we have $\HH^i(G,\HH^4(Z,\Z))\cong(\Z/2\Z)^{44}$.
\end{lem}
\begin{proof}
Statement (i) follows from the K\"unneth formula
$$
\HH^4(S\times S,\Z)=\HH^4(S,\Z)\otimes\HH^0(S,\Z)\oplus
\HH^2(S,\Z)\otimes\HH^2(S,\Z)\oplus \HH^0(S,\Z)\otimes\HH^4(S,\Z).
$$
Then (iii) follows easily from (i) and Lemma \ref{splitting lemma}.
By direct calculation of the group cohomology groups, we get (ii)
and (iv).
\end{proof}

We will use the following spectral sequences frequently.
\begin{prop}[Grothendieck, chapter V of \cite{grothendieck}]\label{spectral sequences}
Let $X$ be a topological space with an action by a finite group $G$.
Let $Y=X/G$ be the quotient and $\pi:X\to Y$ the natural map. Let
$\Gamma^G(-)$ be the covariant functor from the category of
$G$-sheaves on $X$ to the category of abelian groups which sends a
$G$-sheaf $\calF$ to the $G$-invariant sections $\HH^0(X,\calF)^G$.
The right derived functor of $\Gamma^G$ is denoted $\HH^i(G;X,\Z)$.
Then there are two spectral sequences associated to the situation,
$$
_IE_2^{p,q}=\HH^p(Y,R^q(\pi_*^G)\calF)\Longrightarrow\HH^{p+q}(G;X,\calF)
$$
and
$$
_{II}E_2^{p,q}=\HH^p(G,\HH^q(X,\calF))\Longrightarrow\HH^{p+q}(G;X,\calF)
$$
where $\pi^G_*(\calF)=(\pi_*\calF)^G$ for all $G$-sheaves $\calF$.
\end{prop}

Now we apply the second spectral sequence to the special case
$\pi:Z\to F=Z/G$ and $\calF=\Z$. Note that $\HH^i(Z,\Z)=0$ for odd
$i$. This together with Lemma \ref{group cohomology weight 4} force
the spectral sequence to degenerate at the $_{II}E_2$ page. Thus we get the following lemma.
\begin{lem}\label{second spectral sequence lemma}
(i) There is a short exact sequence
$$
\xymatrix{
 0\ar[r] &\Z/2\Z\ar[r] &\HH^2(G;Z,\Z)\ar[r] &\HH^2(Z,\Z)^G\ar[r] &0.
}
$$
(ii) We also have the following short exact sequence
$$
\xymatrix{
 0\ar[r] &T_1\ar[r] &\HH^4(G;Z,\Z)\ar[r] &\HH^4(Z,\Z)^G\ar[r] &0,
}
$$
where $T_1$ is a torsion group of order 4.
\end{lem}

Consider the first spectral sequence $_IE$. Direct calculation of
the stalks gives
$$
R^i\pi^G_*\Z=\begin{cases}
 \Z, &i=0;\\
 0, &i\text{ odd};\\
 (\Z/2\Z)_{\Delta}, &i>0\text{ even}.
\end{cases}
$$
This creates enough zeros in the $_IE_2$ page and forces the
spectral sequence to degenerate at this page. Hence we have the following lemma.
\begin{lem}\label{first spectral sequence lemma}
(i) There is a short exact sequence
$$
\xymatrix{
 0\ar[r] &\HH^2(S^{[2]},\Z)\ar[r] &\HH^2(G;Z,\Z)\ar[r]
 &\HH^0(\Delta,\Z/2\Z)\ar[r] &0.
}
$$
(ii) We also have a short exact sequence
$$
\xymatrix{
 0\ar[r] &\HH^4(S^{[2]},\Z)\ar[r] &\HH^4(G;Z,\Z)\ar[r] &T_2\ar[r] &0,
}
$$
where $T_2$ is torsion group that fits into the following exact
sequence
$$
\xymatrix{
 0\ar[r] &\HH^2(\Delta,\Z/2\Z)\ar[r] &T_2\ar[r]
 &\HH^0(\Delta,\Z/2\Z)\ar[r] &0.
}
$$
In particular, $T_2$ is of order $2^{24}$.
\end{lem}

\begin{cor}[Beauville, \cite{beauville}]\label{size of T3}
(i) The homomorphism $\pi^*:\HH^2(S^{[2]},\Z)\to\HH^2(Z,\Z)^G$ is an
isomorphism and hence we have a canonical isomorphism
\begin{equation}\label{H2}
\HH^2(S^{[2]},\Z)=\HH^2(S,\Z)\oplus \Z\delta
\end{equation}
where $\delta$ satisfies $2\delta=[\Delta]$.\\
(ii) There is a short exact sequence
$$
\xymatrix{
 0\ar[r] &\HH^4(S^{[2]},\Z)\ar[r]^{\pi^*} &\HH^4(Z,\Z)^G\ar[r]
 &T_3\ar[r] &0
}
$$
where $T_3\cong(\Z/2\Z)^{22}$.
\end{cor}
\begin{proof}
To prove (i), we first note that the composition
$$
\HH^2(G,\HH^0(Z,\Z))=\Z/2\Z\rightarrow \HH^2(G;Z,\Z)\rightarrow
\HH^0(\Delta,R^2\pi_*^G\Z)=\Z/2\Z
$$
is an isomorphism. By easy diagram chasing, this forces the composition
$$
\HH^2(S^{[2]},\Z)\rightarrow \HH^2(G;Z,\Z)\rightarrow \HH^2(Z,\Z)^G
$$
to be an isomorphism. Hence there is an element
$\delta\in\HH^2(S^{[2]},\Z)$ such that $\pi^*\delta=[\tilde\Delta]$.
Since $\pi^*[\Delta]=2[\tilde\Delta]$, we get $2\delta=[\Delta]$. To
prove (ii), we consider the following diagram
$$
\xymatrix{
 & &0 &0 &\\
 0\ar[r] &\HH^4(S^{[2]},\Z)\ar[r]^{\pi^*} &\HH^4(Z,\Z)^G\ar[r]\ar[u]
 & T_3\ar[r]\ar[u] &0\\
 0\ar[r] &\HH^4(S^{[2]},\Z)\ar@{=}[u]\ar[r]
 &\HH^4(G;Z,\Z)\ar[r]\ar[u] &T_2\ar[r]\ar[u] &0\\
 & &T_1\ar@{=}[r]\ar[u] &T_1\ar[u] &\\
 & &0\ar[u] &0\ar[u] &
}
$$
This implies that $T_3$ is of order $2^{22}$. For any
$\alpha\in\HH^4(Z,\Z)^G$, we have $2\alpha=\pi^*(\pi_*\alpha)$. This
implies that $T_3$ is of 2-torsion. Hence we have
$T_3\cong(\Z/2\Z)^{22}$.
\end{proof}

\begin{rmk}
For any $\fa\in\HH^2(S,\Z)$, its image in $\HH^2(F,\Z)$ will be
denoted $\hat\fa$. Hence $\HH^2(S^{[2]},\Z)$ has a basis
$$
\{\hat\fa_1,\hat\fa_2,\ldots,\hat\fa_{22},\delta\}.
$$
The element $\hat\fa_i$ can be described as the unique element
satisfying
$$
\pi^*\hat\fa_i=\tau^*(p_1^*\fa_i+p_2^*\fa_i).
$$
\end{rmk}

Consider the following diagram
$$
\xymatrix{
 & &T_3\ar@{=}[r] &T_3 &\\
 0\ar[r] &\Sym^2(\HH^2(Z,\Z)^G)\ar[r]^{\quad\tilde\rho} &\HH^4(Z,\Z)^G\ar[r]\ar[u]
 &T_4\ar[r]\ar[u] &0\\
 0\ar[r]
 &\Sym^2(\HH^2(S^{[2]},\Z))\ar[r]^{\quad\rho}\ar[u]_{\simeq}
 &\HH^4(S^{[2]},\Z)\ar[r]\ar[u]_{\pi^*} &\mathcal{T}^4(S^{[2]})\ar[r]\ar[u] &0
}
$$
where all the 3-terms columns are short exact. Use the canonical
isomorphism in (i) of Lemma \ref{second cohomology of Z}, we know
that the image of $\tilde\rho$ is freely generated by
\begin{align*}
\tau^*(p_1^*\fa_i+p_2^*\fa_i)\cup\tau^*(p_1^*\fa_i+p_2^*\fa_i)
&=\tau^*(p_1^*(\fa_i\cup\fa_i)+p_2^*(\fa_i\cup\fa_i)+
2p_1^*\fa_i\cup
p_2^*\fa_i)\\
&= (\fa_i\cdot\fa_i)e_0 + 2 e_i,\quad 1\leq i\leq 22;\\
\tau^*(p_1^*\fa_i+p_2^*\fa_i)\cup\tau^*(p_1^*\fa_j+p_2^*\fa_j)
&=\tau^*(p_1^*(\fa_i\cup\fa_j)+p_2^*(\fa_i\cup\fa_j)+\\
&\qquad p_1^*\fa_i\cup p_2^*\fa_j+p_1^*\fa_j\cup p_2^*\fa_i)\\
&=(\fa_i\cdot\fa_j)e_0 + e_{ij},\quad 1\leq i<j\leq 22;\\
\tau^*(p_1^*\fa_i+p_2^*\fa_i)\cup[\tilde\Delta] &=
j_*(j^*\tau^*(p_1^*\fa_i+p_2^*\fa_i))\\
&=j_*(2\eta^*\fa_i)\\
&=2u'_i
\end{align*}
and
\begin{equation}\label{Delta tilde squaired}
[\tilde\Delta]^2=j_*j^*[\tilde\Delta]=-j_*\xi
\end{equation}
where $\xi$ is the first Chern class of the relative $\calO(1)$
bundle of $\tilde\Delta=\PP(T_S)\to S$. The next lemma gives us
\begin{equation}\label{push forward of xi}
j_*\xi=j_*(c_1(\mathcal{E}))=j_*(c_1(\mathcal{E})\cap\eta^*[S])=\tau^*[\Delta_S],
\end{equation}
where $\Delta_S\subset S\times S$ is the diagonal and $\mathcal{E}$
fits into the following short exact sequence
$$
\xymatrix{ 0\ar[r] &\calO_{\tilde\Delta}(-1)\ar[r] &\eta^* T_S\ar[r]
&\mathcal{E}\ar[r] &0. }
$$

\begin{lem}[\cite{fulton}, Proposition 6.7]\label{fundamental formla of blow up}
Let $X$ be a smooth projective variety and $Y\subset X$ a smooth
subvariety of codimension $d$. Let $\tau:\tilde{X}\to X$ be the blow
up of $X$ along $Y$ and $\tilde{Y}=\PP(\mathscr{N}_{Y/X})\subset
\tilde{X}$ the exceptional divisor. Let $j:\tilde{Y}\to\tilde{X}$ be
the inclusion, $\eta:\tilde{Y}\to Y$ the natural projection and
$i:Y\to X$ the inclusion. On $\tilde{Y}$, we have the following
short exact sequence
$$
 \xymatrix{ 0\ar[r] &\calO_{\tilde{Y}}(-1)\ar[r] &\eta^*\mathscr{N}_{Y/X}\ar[r] &\mathcal{E}\ar[r]
 &0.
 }
$$
Then we have
$$
\tau^*i_*(x)=j_*(c_{d-1}(\mathcal{E})\cap \eta^*x)
$$
for all $x$ in the cohomology or Chow groups of $Y$.
\end{lem}
Let $A=(a_{ij})_{22\times 22}$, with $a_{ij}=\fa_i\cdot\fa_j\in\Z$,
be the intersection matrix of $\HH^2(S,\Z)$. Since the intersection
form is unimodular, we know that $B=A^{-1}$ is integral. Actually,
let $\{\fa_1^{\vee},\fa_2^\vee,\ldots,\fa_{22}^\vee\}$ be the dual
basis of $\HH^2(S,\Z)^\vee$. The intersection form gives a canonical
isomorphism $\HH^2(S,\Z)^\vee\cong\HH^2(S,\Z)$, under which we have
$$
\fa_i^\vee=\sum_{j=1}^{22}b_{ij}\fa_j.
$$
As a correspondence, $\Delta_S$ acts trivially on the cohomology of $S$;
it follows that
$$
[\Delta_S]=\sum_{i=1}^{22} \fa_i \otimes \fa_i^\vee
+[pt]\otimes[S]+[S]\otimes [pt].
$$
This implies
$$
\tau^*[\Delta_S]=\sum_{1\leq i<j\leq 22}b_{ij}e_{ij}
+\sum_{i=1}^{22}b_{ii}e_i + e_0.
$$
This combined with \eqref{Delta tilde squaired} and \eqref{push
forward of xi} implies that
$$
[\tilde\Delta]^2=-\sum_{1\leq i<j\leq 22}b_{ij}e_{ij}
-\sum_{i=1}^{22}b_{ii}e_i - e_0.
$$
With the chosen basis, the matrix representation for $\tilde\rho$ is
$$
M_{\tilde\rho}=\begin{pmatrix}
 2 &       &  &  &       &  &\vdots &  &       & \\
   &\ddots &  &  &       &  &a_{ii} &  &       & \\
   &       &2 &  &       &  &\vdots &  &       & \\
   &       &  &1 &       &  &\vdots &  &       & \\
   &       &  &  &\ddots &  &a_{ij} &  &       & \\
   &       &  &  &       &1 &\vdots &  &       & \\
 \cdots &-b_{ii} &\cdots &\cdots &-b_{ij} &\cdots &-1 & & &\\
   &       &  &  &       &  &      &2 &       & \\
   &       &  &  &       &  &      &  &\ddots & \\
   &       &  &  &       &  &      &  &       &2\\
\end{pmatrix}
$$
In the matrix $M_{\tilde\rho}$, we first have $22$ of ``2''s on diagonal indexed by $(i,i)$, $1\leq i\leq 22$; then we have 231 of ``1''s
indexed by $(i,j)$, $1\leq i\leq j\leq 22$; after that there is a single ``$-1$'' followed by 22 of ``2''s. In the row
(resp. column) corresponding to the diagonal entry ``$-1$'', we have all the entries $b_{ij}$ of $B$ sitting before (resp. $a_{ij}$ of
$A$ sitting above) ``$-1$''. All the remaining entries of $M_{\tilde\rho}$ are ``0''s. One checks that $\det(M_{\tilde\rho})=5\cdot 2^{45}$. Hence we have
the following lemma.
\begin{lem}\label{size of torsions}
(i) The group $T_4$ is of order $5\cdot 2^{45}$.\\
(ii) The group $\mathcal{T}^4(S^{[2]})$ is of order $5\cdot 2^{23}$.
\end{lem}
Statement (ii) was first obtained in \cite[Proposition 6.6]{bns}. Their method uses the isomorphism between $\HH^2(S^{[2]},\Z)$ and the lattice $U^{\oplus 3}\oplus E_8(-1)^{\oplus 2} \oplus \langle -2 \rangle$ and a discriminant computation.

In \cite[\S 4]{qin-wang}, the operators $L^\lambda$ were introduced to study the cohomology of Hilbert scheme
of $n$ points on a surface, where $\lambda$ is a partition of $n$. In our case, we will use the operator $L^{1,1}$
 which can be described explicitly as
\begin{equation}\label{eq L11}
 L^{1,1}(\fa) = \frac{1}{2}\pi_*(\tau^*(p_1^*\fa \cup p_2^*\fa) - j_*\eta^*\fa) \in \HH^4(F,\Z),
\end{equation}
for all $\fa\in\HH^2(S,\Z)$. The key point here is that $L^{1,1}$ is integral. When $\fa$ is the
class of a curve $C\subset S$, then $L^{1,1}(\fa)$ is represented by
the closure of $\{\pi(x,y):x,y\in C,x\neq y\}$. This implies that $L^{1,1}$ is integral on Hodge classes;
see \cite[Theorem 4.5]{qin-wang}. The above equation \eqref{eq L11} allows us to check the integrality of
 $L^{1,1}$ on a basis of $\HH^2(S,\Z)$. Thus we only need to check that for finitely many general elements
 $\fa$ of $\HH^2(S,\Z)$. For each of those $\fa$, by deforming the complex structure on $S$, we can arrange $\fa$
to be a Hodge class and hence we see that $L^{1,1}$ is integral on this $\fa$. It follows that $L^{1,1}$ is
integral on all of those finitely many $\fa$'s. Hence as an operator, $L^{1,1}$ is integral.

The following theorem will be very useful in explicit
calculations.
\begin{thm}\label{explicit basis}
Let $\{\fa_i\}$ be an integral basis of $\HH^2(S,\Z)$ and $A$ the
intersection matrix. Let $B=A^{-1}$. There is a basis
$$
\{v_0,\,\,v_i|_{1\leq i\leq 22},\,\, v_{ij}|_{1\leq i<j\leq 22},
\,\,u_i|_{1\leq i\leq 22}\}
$$
of $\HH^4(S^{[2]},\Z)$, such that
\begin{equation*}
 \pi^*v_0 =e_0,\quad \pi^*v_{i}= e_i-u'_i,\quad
 \pi^*v_{ij}=e_{ij},\quad \pi^*u_i=2u'_i.
\end{equation*}
The cup product $\HH^2(S^{[2]},\Z)\times\HH^2(S^{[2]},\Z)\to
\HH^4(S^{[2]},\Z)$ can be described explicitly as
\begin{align*}
\hat{\fa}_i\cdot\hat{\fa}_j &=v_{ij}+a_{ij}v_0,\quad 1\leq
i<j\leq22;\\
\hat{\fa}_i\cdot\hat{\fa}_i &=2v_i +u_i+a_{ii}v_0,\quad 1\leq
i\leq22;\\
\hat{\fa}_i\cdot\delta &=u_i,\quad 1\leq i\leq 22;\\
\delta\cdot\delta & =-\sum_{1\leq i<j\leq
22}b_{ij}v_{ij}-\sum_{1\leq i\leq 22}b_{ii}v_i -\sum_{1\leq i\leq
22}\frac{b_{ii}}{2}u_i-v_0.
\end{align*}
In particular, $\hat\fa_i(\hat\fa_i - \delta)$ is divisible by $2$ in $\HH^4(S^{[2]},\Z)$.
\end{thm}
\begin{proof}
We take $u_i=j'_*\eta^*\fa_i=\pi_*u'_i$,
$v_0=\pi_*(\tau^*p_1^*[pt])$ and $v_{ij}=\pi_*\tau^*(p_1^*\fa_i\cup
p_2^*\fa_j)$. We also set $v_i=L^{1,1}(\fa_i)\in\HH^4(S^{[2]},\Z)$. We also have the
following relation
$$
e_i=\tau^*(p_1^*\fa_i\cup p_2^*\fa_i)=\pi^*L^{1,1}(\fa_i)+u'_i,\quad
1\leq i\leq 22.
$$
Let $N\subset\HH^4(S^{[2]},\Z)$ be the subgroup generated by $v_0$,
$\{v_i\}_{i=1}^{22}$, $\{v_{ij}\}_{1\leq i<j\leq 22}$,
$\{u_i\}_{i=1}^{22}$. By writing down explicitly the matrix of
$\pi^*:N\to\HH^4(Z,\Z)^G$, we see that the cokernel of the above map
has size $2^{22}$. By (ii) of Lemma \ref{size of T3}, this implies
that $N=\HH^4(S^{[2]},\Z)$. To check the formula of cup products, we
only need to check the identities after pulling back via $\pi^*$.
For example,
\begin{align*}
\pi^*(\hat\fa_i\cdot\hat\fa_j) &= \tau^*(p_1^*\fa_i+p_2^*\fa_i)\cdot
\tau^*(p_1^*\fa_j+p_2^*\fa_j)\\
&=\tau^*((\fa_i\cdot\fa_j)p_1^*[pt] +(\fa_j\cdot\fa_i)p_2^*[pt])
+\tau^*(p_1^*\fa_i\cup p_2^*\fa_j+ p_1^*\fa_j\cup p_2^*\fa_i)\\
&=\pi^*(a_{ij}v_0+v_{ij}),\quad 1\leq i<j\leq 22.
\end{align*}
The other equalities are checked similarly.
\end{proof}

\subsection{The general case} In this subsection, we assume that $F$ is a hyperk\"ahler manifold of $K3^{[2]}$-type.
For any $\delta\in\Omega(F)$, then $\delta$ (or $-\delta$) essentially arises from some $S^{[2]}$ in the
deformation equivalent family, see Lemma 3.4 of \cite{mm}. Hence for cohomological computations, it is harmless to automatically view an element $\delta\in\Omega(F)$ as
coming from an isomorphism $F\cong S^{[2]}$.

We first prove the following
\begin{lem}\label{lemma on change of delta}
Let $\delta,\delta'\in\Omega(F)$, then the following are true.\\
(i) The class $\delta-\delta'$ is divisible by $2$ in
$\HH^2(F,\Z)$.\\
(ii) The class $\delta^2-\delta'^2$ is divisible by $8$ in
$\HH^4(F,\Z)$. \\
(iii) For any $\alpha\in\HH^2(F,\Z)$, the class
$\alpha\cdot(\alpha-\delta)\in\HH^4(F,\Z)$ is divisible by $2$.
\end{lem}
\begin{proof}
We assume that $F=S^{[2]}$ for some $K3$-surface $S$. Let $\delta$
be the half of the boundary divisor. We use the notations in
Theorem \ref{explicit basis}. Then we can write
$\delta'=\hat\fa'-c\delta$, where $b(\hat\fa',\delta)=0$. Since
$\delta'$ is exceptional, we have $b(\delta',\alpha)\in 2\Z$ for all
$\alpha\in\HH^2(F,\Z)$. This forces $\hat\fa'=2\hat\fa$, for some
$\hat\fa=\sum a_i\hat\fa_i\in\delta^\perp$. Since $\delta'$ is
primitive, we know that $c\in\Z$ is odd. This proves (i). For (ii),
we write explicitly
\begin{align*}
\delta'^2-\delta^2
 &=4\hat\fa^2-4c\delta\hat\fa+c^2\delta^2\\
 &=4\sum_{i=1}^{22}(a_i^2\hat\fa_i^2-ca_i\,\delta\hat\fa_i) +8\sum_{1\leq i<j\leq 22}
 a_ia_j\hat\fa_i\hat\fa_j +(c^2-1)\delta^2.
\end{align*}
Since $c^2-1$ is divisible by 8, we only need to show that the first
sum is divisible by 8. This can be seen from
$$
4\sum_{i=1}^{22}(a_i^2 \hat{\fa}_i^2 - ca_i\,\delta\hat{\fa}_i) = 4\sum
a_i(\hat{\fa}_i^2-\delta\hat{\fa}_i) + 8\sum{a_i \choose 2}\hat{\fa}_i^2 -4(c-1)\sum a_i\delta \hat\fa_i
$$
and the fact that $c$ is odd and that $\hat\fa_i^2-\delta\hat\fa_i$ is divisible by 2,
see Theorem \ref{explicit basis}. Given (i), we only need to prove
(iii) for the exceptional element $\delta\in\Omega$ coming from an
isomorphism $F\cong S^{[2]}$. This case follows from Theorem
\ref{explicit basis}.
\end{proof}

\begin{defn}\label{defn of v_0 and v_delta}
Let $\delta\in\Omega(F)$. We define a quadratic map $v_\delta
:\HH^2(F,\Z)\to \HH^4(F,\Z)$ by putting
\begin{equation}\label{eq v delta}
v_\delta(\alpha)=\frac{1}{2}\alpha\cdot(\alpha-\delta).
\end{equation}
It is known that $\delta^\perp$ is the lattice of a $K3$ surface.
Let $\{\hat\fa_1,\ldots,\hat\fa_{22}\}$ be a basis for
$\delta^\perp$ and let $A=(b(\hat\fa_i,\hat\fa_j))_{22\times 22}$ be
the intersection matrix. Set $B=A^{-1}$. We define
\begin{equation}\label{eq v0}
v_0(\delta)=\frac{1}{10}(\delta\cdot\delta+\frac{1}{2}\sum_{i,j=1}^{22}b_{ij}\hat\fa_i\cdot\hat\fa_j).
\end{equation}
\end{defn}
\begin{rmk}
Let $F=S^{[2]}$ and notations be as in Theorem \ref{explicit basis},
then $v_0(\delta)$ is simply the class $v_0$. This can be deduced from a direct computation using the cup product formulas obtained in Theorem \ref{explicit basis}. In particular, $v_0(\delta)\in\HH^4(F,\Z)$ is integral.
\end{rmk}

\begin{thm}\label{explicit basis II}
Let $\delta\in\Omega(F)$ and $\{\hat\fa_1,\ldots,\hat\fa_{22}\}$ be a basis of $\delta^\perp$. Then the cohomology group $\HH^4(F,\Z)$ admits an integral
basis
$$
\{v_0(\delta),\quad\hat\fa_i\cdot\hat\fa_j|_{1\leq i\leq j\leq 22},\quad
v_\delta(\hat\fa_i)|_{1\leq i\leq
22},\quad\delta\cdot\hat\fa_i|_{1\leq i\leq22}\}.
$$
\end{thm}

\section{A description of $\mathcal{T}^4(F)$}
Let $F$ be a hyperk\"ahler manifold of $K3^{[2]}$-type.
This section is devoted to a canonical description of the group
$\mathcal{T}^4(F)$. By Lemma \ref{size of torsions}(ii), we know that the order of the group $\mathcal{T}^4(F)$ is $5\cdot 2^{23}$.

\begin{lem}\label{lem v_delta}
 Let $F$ be a hyperk\"ahler manifold of $K3^{[2]}$-type and
$\delta\in\Omega(F)$ an exceptional class. Then\\
(i) The composition $\xymatrix{\HH^2(F,\Z)\ar[r]^{v_\delta}
&\HH^4(F,\Z)\ar[r] &\mathcal{T}^4(F)}$ induces a homomorphism
$$
\bar{v}_{\delta}:\HH^2(F,\Z)\otimes\Z/2\Z \rightarrow
\mathcal{T}^4(F).
$$
(ii) The homomorphism $\bar{v}_\delta$ is independent of the choice
of $\delta$.\\
(iii) The image $\bar{\delta}$ of $\delta$ in
$\HH^2(F,\Z)\otimes\Z/2\Z$ is independent of the choice of
$\delta$.\\
(iv) The kernel of $\bar{v}_\delta$ is generated by
$\bar{\delta}$.
\end{lem}
\begin{rmk}
Since $\bar{v}_\delta$ is independent of the choice of $\delta$, we
will simply write
$$
\bar{v}:\HH^2(F,\Z/2\Z)\to\mathcal{T}^4(F)
$$
for this canonical homomorphism.
\end{rmk}
\begin{proof}
 For (i), we note that
\begin{align*}
v_\delta(\hat\fa+\hat\fa')
 &=\frac{1}{2}(\hat\fa+\hat\fa')(\hat\fa+\hat\fa'-\delta)\\
 &=v_\delta(\hat\fa) +v_\delta(\hat\fa')+ \hat\fa\hat\fa'
\end{align*}
which implies that the map $\HH^2(F,\Z)\to\mathcal{T}^4(F)$ is a
homomorphism. One easily checks that this homomorphism vanishes on
$2\HH^2(F,\Z)$ and hence (i) follows. If we pick another
exceptional class $\delta'\in\Omega(F)$, then
$$
v_{\delta'}(\hat\fa)-v_{\delta}(\hat\fa)
=\frac{1}{2}(\delta-\delta')\hat\fa.
$$
By Lemma \ref{lemma on change of delta}, the element
$\frac{\delta-\delta'}{2}$ is integral. Thus $v_{\delta'}(\hat\fa)$
and $v_\delta(\hat\fa)$ map to the same image in $\mathcal{T}^4(F)$,
which proves (ii). Statement (iii) follows from the fact that
$\delta-\delta'$ is divisible by 2; see Lemma \ref{lemma on change
of delta}. By the definition of $\bar{v}_\delta$, we see that
$\bar{\delta}$ is in the kernel. Let
$\{\hat\fa_1,\ldots,\hat\fa_{22}\}$ be a basis of $\delta^\perp$.
Then Theorem \ref{explicit basis} implies that
$\bar{v}_{\delta}(\hat\fa_i)\neq 0$ for all $i$. Since the basis is
arbitrary, we get that $\bar{v}_{\delta}(\hat\fa)\neq 0$ for all
primitive $\hat\fa\in\delta^\perp$. This implies that the images
$\bar{v}_\delta(\hat\fa_i)$ are $\Z/2\Z$-linearly independent. This
proves (iv).
\end{proof}

\begin{lem}\label{lem v_0}
Let $F$ be a hyperk\"ahler manifold of $K3^{[2]}$-type and
$\delta\in\Omega(F)$ an exceptional class. Then the following are true.\\
(i) $v_0(\delta)\in\HH^4(F,\Z)$ is the unique element satisfying the
following
$$
(v_0(\delta)\cdot\alpha\cdot\beta)_F=b(\alpha,\beta)+\frac{1}{4}b(\delta,\alpha)b(\delta,\beta),
\quad \forall\alpha,\beta\in\HH^2(F,\Z).
$$
(ii) The image $\bar{v}_0(\delta)$ of $v_0(\delta)$ in
$\mathcal{T}^4(F)$ is an element of order $10$. For a different
choice $\delta'\in\Omega(F)$, the difference
$\bar{v}_0(\delta)-\bar{v}_0(\delta')$ is an element in the image of
$\bar{v}_{\delta}$. In particular, the element $\bar{w}_0=2\bar{v}_0(\delta)$ is independent of the choice of $\delta$.
\end{lem}

\begin{proof}
We take a basis $\{\hat\fa_1,\ldots,\hat\fa_{22}\}$ of $\delta^{\perp}$. To prove (i), we may work in the special case $F=S^{[2]}$.
Then the $v_0$ in Theorem \ref{explicit basis} is represented by a smooth surface $\tilde{S}\subset
F$, which parameterizes all length two subschemes of $S$ containing a given point of $S$. Note that $\tilde{S}$ is the blow up of $S$ at that point. Then
we have
$$
\hat\fa_i|_{\tilde{S}}=\sigma^*\fa_i,\qquad \delta|_{\tilde{S}}=E
$$
where $\sigma:\tilde{S}\to S$ is the blow up with $E$ being the
exceptional curve. Since the class of $\tilde{S}$ on $F$ is simply $v_0$, we have
$$
(v_0\cdot\hat\fa_i\cdot\delta)_F=(\sigma^*\fa_i\cdot
E)_{\tilde{S}}=0, \qquad (v_0\cdot\delta^2)_F=E^2=-1,
$$
and
$$
(v_0\cdot\hat\fa_i\cdot\hat\fa_j)_F=(\sigma^*\fa_i\cdot\sigma^*\fa_j)_{\tilde{S}}
=(\fa_i\cdot\fa_j)=b(\hat\fa_i,\hat\fa_j).
$$
One easily checks that the equality in (i) holds for $\alpha,\beta$
from the basis $\{\hat\fa_1,\ldots,\hat\fa_{22},\delta\}$. By
linearity, (i) holds for all $\alpha,\beta$. For (ii), we may still assume $F=S^{[2]}$. First we note that by the
definition of $v_0(\delta)$, we know that its image in
$\mathcal{T}^4(F)$ is an element of order 10; see equation \eqref{eq v0}. In $\HH^4(F,\Z)$, we
have the element
$$
q=\sum_{1\leq i, j\leq 22}b_{ij}\hat\fa_i\hat\fa_j
-\frac{1}{2}\delta^2
$$
where $B=(b_{ij})$ is the inverse of the intersection matrix of
$\{\fa_1,\ldots,\fa_{22}\}$. We know that $q$ is independent of the
choice of the exceptional class $\delta$ and the basis
$\{\hat\fa_i\}$ of $\delta^\perp$, see the discussion of next
section. Then equation \eqref{eq v0} can be written as
$$
v_0(\delta)=\frac{1}{8}\delta^2+\frac{1}{20}q.
$$
This implies that
$$
v_0(\delta')-v_0(\delta)=\frac{1}{8}(\delta'^2-\delta^2).
$$
As in the proof of Lemma \ref{lemma on change of delta}, we can
write $\delta'=2\hat\fa+c\delta$ for some $\hat\fa\in \delta^\perp$
and odd integer $c$. Then by direct computation, we get
$$
\bar{v}_0(\delta')-\bar{v}_0(\delta)=\bar{v}_\delta(\hat\fa).
$$
This finishes the proof.
\end{proof}

\begin{thm}\label{thm on T^4}
Let $F$ be a hyperk\"ahler manifold of $K3^{[2]}$-type. Let $\bar{\delta}\in \HH^2(F,\Z/2\Z)$ be the canonical element as in Lemma \ref{lem v_0}(iii) and $\mathcal{K}^2(F)$ be the quotient
of $\HH^2(F,\Z/2\Z)$ by $(\Z/2\Z)\bar{\delta}$. Then the following are true.

(i) There is a
canonical short exact sequence
$$
\xymatrix{0\ar[r] &\mathcal{K}^2(F)\ar[r] &\mathcal{T}^4(F)\ar[r]
&\Z/10\Z\ar[r] &0}.
$$
Each $\delta\in\Omega(F)$ determines a splitting of the above
sequence by the homomorphism $\Z/10\Z\to\mathcal{T}^4(F)$, $1\mapsto
\bar{v}_0(\delta)$.

(ii) There is a canonical exact sequence
$$
 \xymatrix@C=0.5cm{
   0 \ar[r] & \Z/5\Z \ar[rr]^{\varphi} &&\mathcal{T}^4(F) \ar[rr]^{\psi\qquad} && \HH^2(F,\Z)^\vee\otimes\Z/2\Z \ar[r] & 0 }
$$
where $\varphi(1)=\bar{w}_0$ and
$\psi(\theta)=\{\alpha\mapsto(\alpha\cdot\tilde\theta\cdot\delta)_F\mod
2\}$, where $\tilde\theta\in\HH^4(F,\Z)$ is a lifting of
$\theta\in\mathcal{T}^4(F)$ and $\delta\in\Omega(F)$ is some
exceptional class. The composition $\psi\circ\bar{v}_\delta$ is the
homomorphism induced by the Beauville--Bogomolov pairing.
\end{thm}

\begin{proof}
To prove (i), we note that by Lemma \ref{lem v_0} $\mathcal{K}^2(F)$ is identified with
the image of $\bar{v}$. Since $\mathcal{K}^2(F)$ is of order
$2^{22}$ and $\mathcal{T}^4(F)$ is of order $5\cdot 2^{23}$, we know
that the quotient group $\mathcal{T}^4(F)/\mathcal{K}^2(F)$ is of
order 10 and thus isomorphic to $\Z/10\Z$. This gives the short
exact sequence. By the explicit basis obtained in Theorem
\ref{explicit basis II}, we know that the images of $v_0$ and
$\{v_\delta(\hat\fa_i)\}$ generate $\mathcal{T}^4(F)$. This implies
that the element $\bar{v}_0$, which is of order $10$ by Lemma \ref{lem v_0}, gives a splitting of the above short
exact sequence.

To prove (ii), we first check that the definition of $\psi$ is
independent of the choice of $\delta$. The composition
$\psi\circ\bar{v}_\delta$ is induced by
$$
\alpha\mapsto\{\beta\mapsto(\frac{1}{2}\alpha(\alpha-\delta)\cdot\beta\cdot\delta)_F\mod
2\}.
$$
Then we explicitly compute that
$$(\frac{1}{2}\alpha(\alpha-\delta)\cdot\beta\cdot\delta)_F\equiv b(\alpha,\beta)\mod 2.$$
This implies that $\psi\circ\bar{v}_\delta$ is the homomorphism
induced by $b(-,-)$ whose image is
$$
 \{\ell\in\Hom(\HH^2(F,\Z),\Z/2\Z):\ell(\delta)=0,\forall\delta\in\Omega(F)\}.
$$
To show that $\psi$ is surjective, we only need to show that there
exists an element $\theta\in\HH^4(F,\Z)$ such that
$\theta\cdot\delta\cdot\delta\equiv 1\mod 2$. But we can simply take
$\theta=v_0(\delta)$ since $v_0(\delta)\cdot\delta^2=-1$. The
remaining part of (ii) follows from this.
\end{proof}

\section{Minimal Hodge classes and Picard rank}
Let $F$ be a hyperk\"ahler manifold of $K3^{[2]}$-type. In this section we show that the minimal Hodge classes are of special form. Namely it is always contained in the subspace generated by divisors and the Beauville--Bogomolov form. In particular, under Assumption \ref{technical assumption} we show that being of Jacobian type will force the Picard rank of $F$ to jump. As a consequence we see that a very general $F$ is not of Jacobian type. We also show that having a minimal Hodge class is a birational invariant.

\subsection{Some canonical Hodge classes}
Let $(F,\lambda_0)$ be a primitively polarized hyperk\"ahler manifold of
$K3^{[2]}$-type. In this section, we study some canonical element in the group
$$
\mathrm{Hdg}^4(F)=\HH^4(F,\Z)\cap\HH^{2,2}(F)
$$
of Hodge classes in degree 4. First we note that there are two
canonical classes in $\mathrm{Hdg}^4(F)$. The first one is
$(\lambda_0)^2$ and the second one is constructed from the
Beauville--Bogomolov form via linear algebra.

Let $\Lambda$ be a finitely geneated free abelian group with a nondegenerate integral symmetric bilinear form $b:\Lambda \times \Lambda\rightarrow \Z$. Let $\{\fa_1,\ldots,\fa_r\}\subset\Lambda$ be an integral basis of $\Lambda$. Let $A=(a_{ij})_{r\times r}$ be the matrix representing the bilinear form $b$ with respect to the above basis. Namely $a_{ij}=b(\fa_i,\fa_j)$. Let $B=(b_{ij})=A^{-1}$ and define
\begin{equation}\label{eq def of b}
 b^{-1} = \sum_{i,j=1}^{r} b_{ij}\fa_i \fa_j \in \Sym^2(\Lambda)\otimes\Q.
\end{equation}
Note that $B$ is usually not integral unless the bilinear form $b$ is unimodular. Using standard linear algebra, one checks that $b^{-1}$ is independent of the choice of the basis $\{\fa_i\}$. Actually, we can choose $\{\fa_i\}$ to be a basis of $\Lambda_{\Q}=\Lambda\otimes\Q$ (resp. $\Lambda_{\C}=\Lambda\otimes\C$) and extend $b$ linearly to $\Lambda_\Q$ (resp. $\Lambda_\C$), then the expression \eqref{eq def of b} gives the same element $b^{-1}$.

Back to our specific situation, let $b(-,-)$ be the
Beauville--Bogomolov bilinear form on $\HH^2(F,\Z)$. The above construction gives an element $b^{-1}\in \Sym^2(\HH^2(F,\Q))$. By Lemma \ref{lem sym2}, this gives rise to an element $q\in\HH^4(F,\Q)$. We will frequently use the following explicit expression for $q$. Take an exceptional element $\delta\in\Omega(F)$ and let
$$
\{\hat\fa_1,\hat\fa_2,\ldots,\hat\fa_{22}\}
$$
be a basis of $\delta^\perp$. Let $A=(a_{ij})$ be the matrix of
$b$ restricted to $\delta^\perp$, i.e. $a_{ij}=b(\hat\fa_i,\hat\fa_j)$. Let $B=(b_{ij})=A^{-1}$. Then we have
\begin{equation}\label{eq q}
q=\sum_{i,j}b_{ij}\hat\fa_i\hat\fa_j - \frac{1}{2}\delta^2\in\HH^4(F,\Q).
\end{equation}

\begin{defn}\label{defn V lambda}
 Let $V_{\lambda_0}\subset \mathrm{Hdg}^4(F)$ be the subgroup defined by
$$
V_{\lambda_0}= \mathrm{Span}_{\Q}\{\lambda_0^2, q\}\cap \HH^4(F,\Z).
$$
\end{defn}

Recall that $\lambda_0$ is even if $b(\lambda_0,\alpha)$ is even for all $\alpha\in\HH^2(F,\Z)$.

\begin{lem}\label{lemma on even classes}
Let $\lambda_0\in\HH^2(F,\Z)$ be primitive. The following statements are equivalent.\\
(i) The element $\lambda_0$ is even.\\
(ii) $\lambda_0^2-\delta^2$ is divisible by 2 for some
$\delta\in\Omega(F)$.\\
(iii) $\lambda_0^2-\delta^2$ is divisible by 2 for all
$\delta\in\Omega(F)$.\\
(iv) $\lambda_0^2-\delta^2$ is divisible by 8 for some
$\delta\in\Omega(F)$.\\
(v) $\lambda_0^2-\delta^2$ is divisible by 8 for all
$\delta\in\Omega(F)$.\\
(vi) $\bar{v}(\lambda_0)=0$ in $\mathcal{T}^4(F)$.
\end{lem}
\begin{proof}
(i)$\Rightarrow$(ii). Let $\delta\in\Omega(F)$ be some exceptional
class. Fix a basis $\{\hat\fa_1,\ldots,\hat\fa_{22}\}$ of
$\delta^\perp$. Then we can write $\lambda_0=\hat\fa'+c'\delta$ with
$\hat\fa'\in\delta^\perp$ and $c'\in\Z$. The fact that
$b(\lambda_0,\alpha)\in2\Z$ for all $\alpha\in\delta^\perp$ and the
fact that $b(-,-)$ is unimodular on $\delta^\perp$ implies that
$\hat\fa'=2\hat\fa$ for some $\hat\fa\in\delta^\perp$. Since
$\lambda_0$ is primitive, we know that $c'$ must be odd and hence we
can write $\lambda_0=2\hat\fa+(2c+1)\delta$, for some $c\in\Z$. Then
it is easy to see that $\lambda_0^2-\delta^2$ is divisible by 2.

(ii)$\Rightarrow$(i). We still write $\lambda_0=\hat\fa'+c'\delta$.
Then we get
$$\lambda_0^2-\delta^2=\hat\fa'^2+2c'\hat\fa'\delta+(c'^2-1)\delta^2.$$
Using the explicit basis obtained in Theorem \ref{explicit basis},
we see that the above expression is divisible by 2 only if $2\mid
c'^2-1$ and $\hat\fa'=2\hat\fa$ for some $\hat\fa\in\delta^\perp$.
Hence we get $\lambda_0=2\hat\fa+(2c+1)\delta$ for some $c\in\Z$.
This implies that $\lambda_0$ is even.

(ii)$\Rightarrow$(iii) and (iv)$\Rightarrow$(v). This is due to the
fact that $\delta^2-\delta'^2$ is divisible by 8 for all
$\delta,\delta'\in\Omega(F)$, see Lemma \ref{lemma on change of
delta}.

(iii)$\Rightarrow$(ii) and (v)$\Rightarrow$(iv) are automatic.

(ii)$\Rightarrow$(iv). As above, we can write
$\lambda_0=2\hat\fa+(2c+1)\delta$. Hence we get
$$
\lambda_0^2-\delta^2=8v_\delta(\hat\fa)+8(c+1)\delta\hat\fa+4c(c+1)\delta^2,
$$
which is easily seen to be divisible by 8.

(iv)$\Rightarrow$(vi). Again, we can write
$\lambda_0=2\hat\fa+(2c+1)\delta$. Then we easily see that
$\bar{v}(\lambda_0)=\bar{v}_\delta(\lambda_0)=0$ in
$\mathcal{T}^4(F)$.

(vi)$\Rightarrow$(i). Pick some $\delta\in\Omega(F)$. If we write
$\lambda_0=\hat\fa'+c'\delta$ for some $\hat\fa'\in\delta^\perp$ and
$c'\in\Z$, then $\bar{v}(\lambda_0)=\bar{v}(\hat\fa')=0$ implies
that $\hat\fa'$ is $0$ in $\delta^\perp\otimes\Z/2\Z$, i.e.
$\hat\fa'=2\hat\fa$ for some $\hat\fa\in\delta^\perp$. Since
$\lambda_0$ is primitive, we get $c'=2c+1$ for some $c\in\Z$. This
shows that $\lambda_0$ is even.
\end{proof}

\begin{lem}\label{lemma on integral classes}
Let $(F,\lambda_0)$ be primitively polarized as above, then the
following are true. \\
(i) The two elements $(\lambda_0)^2$ and $q$ are linearly
independent in $\HH^2(F,\C)$.\\
(ii) For any $\alpha,\beta\in\HH^2(F,\Z)$, we have
$$
(q\cdot\alpha\cdot\beta)_F=25 b(\alpha,\beta).
$$
(iii) The element $\frac{2}{5}q$ is integral and primitive, namely
$\frac{2}{5}q\in\HH^4(F,\Z)$. Furthermore,
$$\frac{2}{5}q+\delta^2=8v_0(\delta),\quad\forall\delta\in\Omega(F).$$
(iv) If $\lambda_0$ is odd, then
$$
V_{\lambda_0}=\Z(\lambda_0)^2\oplus\Z(\frac{2}{5}q).
$$
(v) If $\lambda_0$ is even, then $(\lambda_0)^2+\frac{2}{5}q$ is
divisible by 8 and
$$
V_{\lambda_0}=\Z(\lambda_0)^2\oplus\Z\frac{1}{8}(\lambda_0^2+\frac{2}{5}q).
$$
\end{lem}
\begin{proof}
We fix an exceptional class $\delta\in\Omega(F)$ and let $\hat\fa_i$, $i=1,\ldots, 22$, be an integral basis of $\delta^\perp$. Let $A=(a_{ij})_{1\leq i,j\leq 22}$ be the matrix representing the restriction of $b$ to $\delta^\perp$. Let $B=(b_{ij})=A^{-1}$. In particular, we have the explicit expression \eqref{eq q} for the class $q$.

We prove (ii) first by direct computation as follows. Assume that
$$
\alpha=a\delta + \sum x_i\hat\fa_i,\qquad \beta=b\delta + \sum y_i\hat\fa_i.
$$
Then we have
\begin{align*}
(q\cdot\alpha\cdot\beta)_F &=\sum
b_{ij}(\hat\fa_i\cdot\hat\fa_j\cdot\alpha\cdot\beta)_F - \frac{1}{2}(\delta^2\cdot\alpha\cdot\beta)\\
&=\sum
b_{ij}(b(\hat\fa_i,\hat\fa_j)b(\alpha,\beta)+b(\hat\fa_i,\alpha)b(\hat\fa_j,\beta)+
b(\hat\fa_i,\beta)b(\hat\fa_j,\alpha))\\
&\quad -\frac{1}{2}b(\delta,\delta)b(\alpha,\beta) - b(\alpha,\delta)b(\beta,\delta)\\
&= (\sum b_{ij}a_{ij})b(\alpha,\beta) + 2\sum
x_ka_{ki}b_{ij}a_{jl}y_l + b(\alpha,\beta) - 4ab\\
&= \mathrm{tr}(AB)b(\alpha,\beta) +b(\alpha,\beta) +2(\sum x_ka_{kl}y_l  -2ab)\\
&= 22b(\alpha,\beta)+b(\alpha,\beta)+2b(\alpha,\beta)\\
&=25b(\alpha,\beta).
\end{align*}

To prove (i) we first note that for all
$\alpha,\beta\in(\lambda_0)^\perp$, we have
$$
(\lambda_0^2\cdot\alpha\cdot\beta)_F=b(\lambda_0,\lambda_0)b(\alpha,\beta).
$$
Hence if $\lambda_0^2$ and $q$ are proportional, then
$\lambda_0^2=\frac{b_0}{25}q$ where $b_0=b(\lambda_0,\lambda_0)$.
This implies
$$
3b_0^2=(\lambda_0^2\cdot\lambda_0^2)=\frac{b_0}{25}(q\cdot\lambda_0^2)=\frac{b_0}{25}\cdot25b_0=b_0^2.
$$
This forces $b_0=0$ which is impossible since $\lambda_0$ is an
ample class.

To prove (iii), we first compare the equation \eqref{eq v0} and the equation \eqref{eq q} and get
%\begin{align*}
%q &=\sum_{i,j}b_{ij}\hat\fa_i\hat\fa_j-\frac{1}{2}\delta^2\\
% &=\sum_{i\neq j}b_{ij}(v_{ij}+a_{ij}v_0) +\sum_{i=1}^{22} b_{ii}(2v_i
% +u_i+a_{ii}v_0)\\
% &\qquad+\frac{1}{2}(\sum_{i<j}b_{ij}v_{ij}+\sum_{i=1}^{22}(b_{ii}v_i+\frac{b_{ii}}{2}u_i)
% +v_0)\\
% &=\frac{5}{2}(\sum_{i<j}b_{ij}v_{ij}+\sum b_{ii}v_i)
% +\frac{5}{2}\sum\frac{b_{ii}}{2}u_i+\frac{45}{2}v_0
%\end{align*}
\begin{equation*}
5v_0(\delta) + q =\frac{5}{4} \sum_{i,j}b_{ij}\hat\fa_i\hat\fa_j.
\end{equation*}
Note that $\frac{1}{2}\sum b_{ij}\hat\fa_i\hat\fa_j$ is integral. It follows that $$\frac{2}{5}q = \frac{1}{2}\sum_{i,j}b_{ij}\hat\fa_i\hat\fa_j - 2v_0(\delta)$$ is integral and primitive. Equation \eqref{eq v0} implies that $\frac{2}{5}q+\delta^2=8v_0(\delta)$. This proves (iii).

The intersection matrix of $\{\lambda_0^2,\frac{2}{5}q\}$ is given
by
\begin{equation}
M=\begin{pmatrix}
 \lambda_0^2\cdot\lambda_0^2 &\frac{2}{5}q\cdot\lambda_0^2\\
 \frac{2}{5}q\cdot\lambda_0^2 &\frac{2}{5}q\cdot\frac{2}{5}q
\end{pmatrix}
=
\begin{pmatrix}
 3b_0^2 &10b_0\\
 10b_0 &92
\end{pmatrix}
\end{equation}
Hence $\det(M)=176b_0^2$. Now we assume that
$$
\frac{1}{p}(a\lambda_0+b(\frac{2}{5}q))\in\HH^4(F,\Z)
$$
where $p$ is a prime and $a,b\in\Z$ with $\gcd(a,b)=1$. By Theorem \ref{explicit basis II},
we have an integral basis
$$
\mathcal{B}=\{v_0(\delta),v_\delta(\hat\fa_i),\hat\fa_i\hat\fa_j,\delta\hat\fa_i\}
$$
for $\HH^4(F,\Z)$. We assume that $\lambda_0=k\hat\fa_0-c\delta$,
for some $k,c\in\Z$ and primitive $\hat\fa_0\in\delta^\perp$. Let $\mu=b(\lambda_0,\delta)=2c$ and
$\mu_i=b(\lambda_0,\hat\fa_i)=k\tilde{\mu}_i$, where
$\tilde{\mu}_i=b(\hat\fa_0,\hat\fa_i)$. The intersection of
$\{\lambda_0^2,\frac{2}{5}q\}$ with the basis $\mathcal{B}$ is given
by
\begin{equation}\label{intersection with basis}
 \begin{pmatrix}
  \lambda_0^2\\
  \frac{2}{5}q
 \end{pmatrix}
  \cdot\mathcal{B} =
 \begin{pmatrix}
  k^2a_0 -c^2 &\mu_i(\mu_i-\mu)+\frac{a_{ii}}{2}b_0 &2\mu_i\mu_j+b_0a_{ij} &2\mu\mu_i\\
  9 &5a_{ii} &10a_{ij} &0
 \end{pmatrix}
\end{equation}
where $a_0 = b(\hat\fa_0,\hat\fa_0)$.

\vspace{1mm}
\textbf{Claim}: The fact that $p\mid a\lambda_0^2+b(\frac{2}{5}q)$
implies that $p\mid 4kc$.

\textit{Proof of claim.} Since both $\lambda_0^2$ and
$\frac{2}{5}q$ are primitive, we know that $p\nmid ab$. The last
column of \eqref{intersection with basis} implies that $p\mid
4kc\tilde{\mu_i}$, for all $i=1,\ldots,22$. This forces $p\mid 2kc$.\qed

\vspace{1mm}
\textit{Case 1}: $p\mid c$.

In this case we have
$$
b_0=b(\lambda_0,\lambda_0)=k^2\;b(\hat\fa_0,\hat\fa_0) -2c^2\equiv
k^2a_0,\mod p.
$$
The first column of \eqref{intersection with basis} implies
$$
ak^2a_0+9b\equiv 0\mod p.
$$
By looking at the second column of \eqref{intersection with basis},
we get the following
\begin{align*}
0 &\equiv a\mu_i(\mu_i-\mu) +a\cdot\frac{a_{ii}}{2}b_0 +b\cdot 5a_{ii}\mod p\\
 &\equiv
 ak^2\tilde{\mu}_i^2+a\cdot\frac{a_{ii}}{2}k^2a_0
+5a_{ii}b\mod p\\
 &=ak^2\tilde{\mu_i}^2 +\frac{a_{ii}}{2}(ak^2a_0+9b)
 +\frac{a_{ii}}{2}b\\
 &\equiv ak^2\tilde{\mu}_i^2+\frac{a_{ii}}{2}b\mod p.
\end{align*}
This implies that $a_{ii}=c_0\tilde{\mu}_i^2\mod p$, where
$c_0=-\frac{2ak^2}{b}$. Note that here we use the fact that $p\nmid b$. By looking at the third column of
\eqref{intersection with basis}, we get
\begin{align*}
0 &\equiv a\cdot 2k^2\tilde{\mu}_i\tilde{\mu}_j +a\cdot
k^2a_0 a_{ij} +10ba_{ij}\mod p\\
 &=2ak^2\tilde{\mu}_i\tilde{\mu}_j
 +a_{ij}(ak^2a_0+9b) +ba_{ij}\\
 &\equiv 2ak^2\tilde{\mu_i}\tilde{\mu}_j +ba_{ij}\mod p.
\end{align*}
It follows that $a_{ij}=c_0\tilde{\mu}_i\tilde{\mu}_j\mod p$ for all
$i<j$. Combine the above expressions for $a_{ij}$, we have
$$
\pm 1=\det(A)\equiv (c_0)^{22}\det(\tilde{\mu}_i\tilde{\mu}_j)=0\mod
p.
$$
This is a contradiction and hence case 1 never happens.

\textit{Case 2}: $p\mid k$.

In this case, we first have $b_0\equiv -2c^2\mod p$. The first
column of \eqref{intersection with basis} gives
\begin{equation*}
-ac^2 +9b\equiv 0\mod p.
\end{equation*}
Then the third column of \eqref{intersection with basis} gives
\begin{align*}
0 &\equiv 2a\mu_i\mu_j +ab_0a_{ij}+10ba_{ij}\mod p\\
 &\equiv a(-2c^2)a_{ij} +10ba_{ij}\mod p\\
 &= 2a_{ij}(-ac^2+9b)-8ba_{ij}\\
 &\equiv -8ba_{ij}\mod p.
\end{align*}
Since $p\nmid b$, we have $p\mid 8a_{ij}$ for all $i,j$. Since
$\det(A)=\pm 1$, there exists $(i,j)$ such that $p\nmid a_{ij}$. It
follows that $p\mid 8$ and hence $p=2$.

\textit{Case 3}: $p=2$.

In this case, we have $2\mid \lambda_0^2+\frac{2}{5}q$ and also
$8\mid \delta^2+\frac{2}{5}q$. By taking the difference, we see that
$\lambda_0^2-\delta^2$ is divisible by 2. By Lemma \ref{lemma on
even classes}, we know that $\lambda_0$ is even. This implies that
$$
\lambda_0^2+\frac{2}{5}q = (\delta^2+\frac{2}{5}q) +
(\lambda_0^2-\delta^2)
$$
is divisible by 8. When $\lambda_0$ is even, we easily see that
$b_0/2=b(\lambda_0,\lambda_0)/2$ is odd. Hence the power of the
factor 2 in $\det(M)=176b_0^2$ is 64. It follows that
$$
\frac{1}{8}(\lambda_0^2+\frac{2}{5}q)
$$
is no longer divisible by 2 and hence is primitive. This proves (iv)
and (v).
\end{proof}

\begin{cor}\label{cor index 2}
If $\lambda_0$ satisfies the Assumption \ref{technical assumption}, then for all
$\alpha\in V_{\lambda_0}$ and
$\hat\fa,\hat\fa'\in(\lambda_0)^\perp\subset\HH^2(F,\Z)$, we have
$$
(\alpha\cdot\hat\fa\cdot\hat\fa')_F=e\,b(\hat\fa,\hat\fa'),
$$
for some $e\in 2\Z$. In particular, $V_{\lambda_0}$ contains no minimal Hodge classes.
\end{cor}

\begin{proof}
 We note that for all
$\fa,\fa'\in\HH^2(F,\Z)_\mathrm{tr}$, we have
$$
\frac{2}{5}q\cdot\fa\cdot\fa'=10b(\fa,\fa'),\quad\lambda_0^2\cdot\fa\cdot\fa'=b(\lambda_0,\lambda_0)b(\fa,\fa').
$$
Then one concludes by evaluating $\fa\fa'$ on the generators of
$V_{\lambda_0}$.
\end{proof}

\subsection{Minimal Hodge classes}
\begin{prop}\label{minimal hodge class}
Let $F$ be a hyperk\"ahler manifold of $K3^{[2]}$-type. Then a
minimal Hodge class $\theta$ is always in the $\Q$-vector space
spanned by $\Sym^2(\Pic(F))$ and $q$.
\end{prop}

\begin{proof}%[Proof of Proposition \ref{minimal hodge class}]
Let $\{\fa_1,\ldots,\fa_r\}$ be a basis of $\mathrm{Hdg}^2(F)$ and
$\{\mathfrak{t}_1,\ldots,\mathfrak{t}_s\}$ be an orthonormal basis
of $\mathrm{Hdg}^2(F)^\perp\otimes\C$, i.e.
$b(\mathfrak{t}_\alpha,\mathfrak{t}_\beta)=\delta_{\alpha\beta}$,
$\forall \alpha,\beta\in\{1,\ldots,s\}$. Assume that
$\theta\in\HH^4(F,\Z)$ is a minimal Hodge class. We can write
$\theta$ explicitly as
$$
\theta=\sum_{i,j=1}^r x_{ij}\fa_i\fa_j
+\sum_{i=1}^r\sum_{\alpha=1}^s y_{i\alpha}\fa_i\mathfrak{t}_\alpha
+\sum_{\alpha,\beta=1}^s
z_{\alpha\beta}\mathfrak{t}_{\alpha}\mathfrak{t}_{\beta},
$$
for some $x_{ij},y_{i\alpha},z_{\alpha\beta}\in\C$ with
$x_{ij}=x_{ji}$ and $z_{\alpha\beta}=z_{\beta\alpha}$. Since
$\theta$ is a minimal class, for all
$\alpha,\beta\in\{1,\ldots,s\}$, we have
\begin{align*}
\delta_{\alpha\beta} & = \theta\cdot\mathfrak{t}_\alpha\cdot\mathfrak{t}_{\beta}\\
   & =\sum_{i,j}x_{ij}b(\fa_i,\fa_j)\delta_{\alpha\beta} +\sum_{\alpha',\beta'}z_{\alpha'\beta'}(\delta_{\alpha'\beta'}\delta_{\alpha\beta}+\delta_{\alpha'\alpha}\delta_{\beta'\beta}+\delta_{\alpha'\beta}\delta_{\beta'\alpha})\\
   & =\sum_{ij}x_{ij}b(\fa_i,\fa_j)\delta_{\alpha\beta} + (\sum_{\alpha'}z_{\alpha'\alpha'})\delta_{\alpha\beta} +2z_{\alpha\beta}.
\end{align*}
When $\alpha\neq\beta$, we see that $z_{\alpha\beta}=0$; when
$\beta=\alpha$, we get that $z_{\alpha\alpha}=c$ is a constant that
is independent of $\alpha$. By definition, the element $q$ has the
form
$$
q=\sum_{i,j}c_{ij}\fa_i\fa_j +\sum_{\alpha}\mathfrak{t}_\alpha\mathfrak{t}_{\alpha},
$$
which implies that
$q_{\mathrm{tr}}=\sum_{\alpha}\mathfrak{t}_\alpha\mathfrak{t}_{\alpha}$
is of type $(2,2)$ in the Hodge decomposition. Now we take an
integral basis
$\{\tilde{\mathfrak{t}}_1,\ldots,\tilde{\mathfrak{t}_s}\}$ of
$\HH^2(F,\Z)_{\mathrm{tr}}$. Hence we have a new explicit expression
of $\theta$,
$$
\theta=\sum_{i,j=1}^r x_{ij}\fa_i\fa_j
+\sum_{i=1}^r\sum_{\alpha=1}^s
\tilde{y}_{i\alpha}\fa_i\tilde{\mathfrak{t}}_\alpha +
cq_{\mathrm{tr}}
$$
where $\tilde{y}_{i\alpha},c\in\Q$.  Since $\theta$ is a Hodge
class, we have
$$
\sum_{i=1}^r\sum_{\alpha=1}^s
\tilde{y}_{i\alpha}\fa_i\tilde{\mathfrak{t}}_\alpha^{0,2}=0.
$$
This implies that
$$
\sum_{\alpha=1}^s \tilde{y}_{i\alpha}\tilde{\mathfrak{t}}_\alpha^{0,2}=0,
$$
since $\fa_i$ form a basis. Now it follows that $\sum_{\alpha=1}^s
\tilde{y}_{i\alpha}\tilde{\mathfrak{t}}_\alpha$ is of type $(1,1)$
which forces it to be 0. Hence we have $\theta=\sum_{i,j=1}^r
x_{ij}\fa_i\fa_j+cq_{\mathrm{tr}}$ which is in the $\Q$-span of
$\Sym^2(\Pic(F))$ and $q$.
\end{proof}

\begin{cor}\label{jacobian picard rank}
Let $(F,\lambda_0)$ be a primitively polarized hyperk\"ahler
manifold of $K3^{[2]}$-type with $\lambda_0$ satisfying Assumption \ref{technical
assumption}. If $F$ is of Jacobian type, then the Picard number of
$F$ is at least 2. In particular, a very general such $F$ is not of Jacobian type.
\end{cor}

\begin{proof}
If $\Pic(F)=\Z\lambda_0$, then by the proposition a minimal Hodge
class is always in $V_{\lambda_0}=\Span_\Q\{q,\lambda_0^2\}\cap\HH^{2,2}(F)$.
However, by Corollary \ref{cor index 2}, there is no
such minimal Hodge class.
\end{proof}

\begin{prop}\label{birational invariant}
Let $F_1$ and $F_2$ be hyperk\"ahler manifolds of $K3^{[2]}$-type.
If $F_1$ is birational to $F_2$, then $F_1$ has a minimal Hodge
class if and only if $F_2$ does.
\end{prop}

\begin{rmk}
This Proposition suggests that being of Jacobian type is likely to
be a birational invariant.
\end{rmk}

\begin{proof}%[Proof of Proposition \ref{birational invariant}]
Since $F_1$ and $F_2$ are birational to each other, we have have an isomorphism
$$
\HH^2(F_1,\Z)\cong\HH^2(F_2,\Z)
$$
which is compatible with the Beauville--Bogomolov bilinear forms; see
\cite{huybrechts}. Then Theorem \ref{explicit basis II} implies that
$$
\HH^4(F_1,\Z)\cong\HH^4(F_2,\Z)
$$
which is compatible with the intersections. Then the Proposition
follows easily.
\end{proof}

\section{Hodge classes in degree 4}
Let $F$ be a hyperk\"ahler fourfold of $K3^{[2]}$-type. In this section, we study Hodge classes
in $\HH^4(F)$ for a generic $F$ using a deformation argument. It is proved that all Hodge classes are
generated by the polarization and the Beauville--Bogomolov form when $F$ is very general. This provides a
second proof of the fact that a very general $F$ is not of Jacobian type. Since the Beauville--Bogomolov form
is algebraic, our argument also gives a proof of the Hodge conjecture for very general $F$. In the special case
when $F$ is the variety of lines on a very general cubic fourfold, we are able to prove the integral Hodge
conjecture.

\begin{thm}\label{thm on hodge classes}
Let $(F,\lambda_0)$ be a very general primitively polarized
hyperk\"ahler manifold of $K3^{[2]}$-type. Then the group of integral Hodge classes in degree 4 can be expressed
as
$$
\mathrm{Hdg}^4(F)=V_{\lambda_0}.
$$
In particular, if $\lambda_0$ is odd, then $\mathrm{Hdg}^4(F)$ is
freely generated by $\lambda_0^2$ and $\frac{2}{5}q$; if $\lambda_0$
is even, then $\mathrm{Hdg}^4(F)$ is freely generated by
$\lambda_0^2$ and
$\frac{1}{8}(\lambda_0^2+\frac{2}{5}q)$.
\end{thm}

\begin{rmk}
 Given Corollary \ref{cor index 2}, the above theorem gives another proof of the fact that if $\lambda_0$ satisfies Assumption \ref{technical assumption} then a very general $F$ is not of Jacobian type.
\end{rmk}

\begin{proof}
Let $\Lambda=\HH^2(F,\Z)$. We will use $\Lambda_{\C}$ to denote $\Lambda\otimes\C$. Let $\mathcal{P}$ be the period map
$$
\mathcal{P}([F])=\C\omega_F\in\PP(\Lambda_\C)
$$
where $\omega_F\in\HH^0(F,\Omega^2_F)$ is a generator. By the Local
Torelli Theorem (see \cite{beauville}), we know that a deformation
of $(F,\lambda_0)$ corresponds to a deformation of $\omega_F$ such
that $\lambda_0\in\Lambda^{1,1}$. Consider
$$
Q_{\lambda_0}^+=\{\omega\in\PP(\Lambda_\C):b(\omega,\omega)=0,b(\omega,\overline{\omega})>0,
b(\lambda_0,\omega)=0\}.
$$
Assume that the period of $(F,\lambda_0)$ is $\omega$. Then we can
easily compute
$$
T_{[\omega],Q_{\lambda_0}^+}=\{\varphi\in\Hom_\C(\C\omega,
\Lambda_\C/\C\omega):b(\omega,\varphi(\omega))=0,b(\lambda_0,\varphi(\omega))=0\}.
$$
By definition, $\varphi(\omega)$ is only a class in
$\Lambda_\C/\C\omega$. Then we define $b(\omega,\varphi(\omega))$ to
be $b(\omega,\omega')$ for any element $\omega'\in\HH^2(F,\Z)$
representing the class $\varphi(\omega)$. Note that this is well
defined since $b(\omega,\omega)=0$. We define
$b(\lambda_0,\varphi(\omega))$ in a similar way. We have the Hodge
decomposition
$$
\Lambda_\C=\Lambda^{2,0}\oplus\Lambda^{1,1}\oplus\Lambda^{0,2}.
$$
Let $\{\alpha_1,\ldots,\alpha_{21}\}$ be a $\C$-basis of
$\Lambda^{1,1}$. Then we can write
$$
 \varphi(\omega)=\sum
 \mu_i\alpha_i+\mu_0\bar{\omega},\quad\mu_i,\mu_0\in\C.
$$
The condition $b(\omega,\varphi(\omega))=0$ implies that $\mu_0=0$.
The condition $b(\lambda_0,\varphi(\omega))=0$, together with its
equivalent form $b(\lambda_0,\overline{\varphi(\omega)})=0$, implies
that $\varphi(\omega)\in\lambda_0^\perp\cap\Lambda^{1,1}$. Hence we
get the following identification
$$
T_{[\omega],Q_{\lambda_0}^+}=\lambda_0^\perp\cap\Lambda^{1,1}.
$$
When $[\omega]$ moves in $Q_{\lambda_0}^+$ along the direction
$\varphi\in T_{[\omega],Q_{\lambda_0}^+}$, the space $\Lambda^{1,1}$
moves in the Grassmannian $G(21,\Lambda_\C)$ along the direction $\psi$, where
$$
\psi\in\Hom_\C(\Lambda^{1,1},\C\omega\oplus\C\bar{\omega})
$$
is the unique element satisfying
\begin{align*}
b(\psi(\alpha_i),\omega)+b(\alpha_i,\varphi(\omega)) &=0,\\
b(\psi(\alpha_i),\bar\omega)+b(\alpha_i,\overline{\varphi(\omega)})
&=0.
\end{align*}
If we write $\psi(\alpha_i)=a_i\omega+b_i\bar\omega$, then the above
conditions imply
\begin{equation*}
 a_i=-b(\alpha_i,\overline{\varphi(\omega)}),\quad
 b_i=-b(\alpha_i,\varphi(\omega)).
\end{equation*}
When $[\omega]$ moves in $Q_{\lambda_0}^+$ along the direction
$\varphi$, the space $\HH^{2,2}(F)$ moves in the Grassmannian
$G(232,\Sym^2(\Lambda_\C))$ along some direction
$\tilde{\psi}\in\Hom_\C(\HH^{2,2},\Sym^2(\Lambda_\C)/\HH^{2,2})$.
Under the natural identification
$$
\HH^{2,2}=\Sym^2(H^{1,1})\oplus\C\omega\wedge\bar\omega,
$$
the homomorphism $\tilde{\psi}$ is given by
\begin{align*}
\tilde{\psi}(\alpha_i\alpha_j) &=\psi(\alpha_i)\wedge\alpha_j+\alpha_i\wedge\psi(\alpha_j),\\
\tilde{\psi}(\omega\wedge\bar\omega)
&=\varphi(\omega)\wedge\bar\omega+\omega\wedge\overline{\varphi(\omega)}.
\end{align*}
We define the fixed part $\mathrm{Fix}(\HH^{2,2})$ of $\HH^{2,2}$ to
be the set of all $x\in\HH^{2,2}$ such that for all $\tilde{\psi}$
associated to some $\varphi\in T_{[\omega],Q_{\lambda_0}^+}$, we
have $\tilde{\psi}(x)=0$. Our next step is to determine
$\mathrm{Fix}(\HH^{2,2})$ explicitly. Let
$\varphi(\omega)=\sum\mu_i\alpha_i$ and $A=(a_{ij})$ be the
``intersection" matrix on $\HH^{1,1}$, i.e.
$a_{ij}=b(\alpha_i,\alpha_j)$. Let $T=(t_{ij})$ be the matrix
representing the complex conjugation on $\HH^{1,1}$, i.e.
$$
\bar\alpha_i=\sum_{j=1}^{21} t_{ij}\alpha_j.
$$
Then we get
\begin{align*}
 a_i &=-b(\alpha_i,\overline{\varphi(\omega)})\\
  &=-b(\alpha_i,\sum\bar\mu_j\bar\alpha_j)\\
  &=-\sum\bar\mu_j t_{jk}b(\alpha_i,\alpha_k)\\
  &=-\sum_{j,k}a_{ik}t_{jk}\bar\mu_j,\\
 b_i &=-b(\alpha_i,\varphi(\omega))=-\sum a_{ij}\mu_j.
\end{align*}
Let $\mathbf{a}=(a_1,\ldots,a_{21})^t$ and
$\mathbf{b}=(b_1,\ldots,b_{21})^t$. Then the above equations can be
written as
$$
\mathbf{a}=-A\,^tT\overline{\mu}, \qquad \mathbf{b}=-A\mu,
$$
where $\mu=(\mu_1,\ldots,\mu_{21})^t$. Assume that
$$
x=\sum_{i,j}c_{ij}\alpha_i\alpha_j+ c_0\omega\wedge\bar\omega\in
\mathrm{Fix}(\HH^{2,2}),
$$
where $C=(c_{ij})$ is a symmetric matrix. Then we have
\begin{align*}
\tilde{\psi}(x) &=\sum_{i,j}c_{ij}\tilde\psi(\alpha_i\alpha_j) +
c_0\tilde\psi(\omega\wedge\bar\omega)\\
 &=\sum_{ij}c_{ij}(\alpha_i\wedge\psi(\alpha_j)+\psi(\alpha_i)\wedge\alpha_j)
  +c_0(\varphi(\omega)\wedge\bar\omega+\omega\wedge\overline{\varphi(\omega)})\\
 &=2(^t\alpha C\psi(\alpha)) + c_0(\varphi(\omega)\wedge\bar\omega + \omega\wedge\overline{\varphi(\omega)}
  )\\
 &=2 [\,^t\alpha C(\mathbf{a}\omega +\mathbf{b}\bar{\omega})] + c_0(\,^t\alpha\mu\wedge\bar\omega+
 \,^t\bar\alpha\bar\mu\wedge\omega)\\
 &=\,^t\alpha(-2CA(^tT)\bar\mu +c_0(^tT)\bar\mu)\omega +
 \,^t\alpha(-2CA\mu+c_0\mu)\bar\omega
\end{align*}
where $\alpha=(\alpha_1,\ldots,\alpha_{21})^t$ and
$\psi(\alpha)=(\psi(\alpha_1),\ldots,\psi(\alpha_{21}))^t$. Hence we
get
\begin{align*}
x\in\mathrm{Fix}(\HH^{2,2}) &\Leftrightarrow \begin{cases}
  (c_0-2CA)\mu=0\\
  (c_0-2CA)^tT\bar\mu=0
 \end{cases}\forall \varphi(\omega)\in\lambda_0^\perp\cap\Lambda^{1,1}\\
 &\Leftrightarrow (c_0-2CA)\mu=0,\quad\forall \varphi(\omega)\in\lambda_0^\perp\cap\Lambda^{1,1}.
\end{align*}
We write $\lambda_0=\,^t\alpha \mathbf{s}$, where
$\mathbf{s}=(s_1,\ldots,s_{21})$, then
$$
\varphi(\omega)\in \lambda_0^\perp\cap\Lambda^{1,1} \Leftrightarrow
\;^t\mathbf{s}A\mu=0.
$$
Hence we have the following equivalences
\begin{align*}
x\in\mathrm{Fix}(\HH^{2,2}) &\Leftrightarrow
(c_0-2CA)\mu=0,\quad\forall\mu \text{ with }^t\mathbf{s}A\mu=0\\
 &\Leftrightarrow c_0-2CA=\mathbf{t}\,^t\mathbf{s}A,\text{ for some
 }\mathbf{t}=(t_1,\ldots,t_{21})^t\\
 &\Leftrightarrow
 C=\frac{1}{2}c_0A^{-1}-\frac{1}{2}\mathbf{t}\,^t\mathbf{s}.
\end{align*}
Since $C$ is symmetric, the vector $\mathbf{t}$ is a multiple of
$\mathbf{s}$. Hence $x\in\mathrm{Fix}(\HH^{2,2})$ if and only if
$$
C=c_1B+c_2\mathbf{s}\,^t\mathbf{s},\quad c_0=2c_1,
$$
for some $c_1,c_2\in\C$, where $B=A^{-1}$, i.e.
\begin{align*}
x &=\,^t\alpha(c_1B+c_2\mathbf{s}\,^t\mathbf{s})\alpha +
2c_1\omega\wedge\bar\omega\\
&=c_1(\,^t\alpha B\alpha+2\omega\wedge\bar\omega) +c_2\lambda_0^2\\
&=c_1q+c_2\lambda_0^2.
\end{align*}
Hence we get $\mathrm{Fix}(\HH^{2,2})=\Span_\C\{q,\lambda_0^2\}$.
This computation implies that on a very general deformation of
$(F,\lambda_0)$, the classes that remain of $(2,2)$-type are
generated by $q$ and $\lambda_0^2$. This implies that for a very
general $(F,\lambda_0)$, we have
$$
\mathrm{Hdg}^4(F)=\Span_\Q\{q,\lambda_0^2\}\cap\HH^4(F,\Z).
$$
Together with Lemma \ref{lemma on integral classes}, this proves
the theorem.
\end{proof}

\begin{cor}
Let $(F,\lambda_0)$ be very general as in the theorem, then
$\lambda_0$ is odd (resp. even) if and only if the image of
$\mathrm{Hdg}^4(F)$ in $\mathcal{T}^4(F)$ is a cyclic group of order
5 (resp. 10).
\end{cor}
\begin{proof}
In any case, we always have $\lambda_0^2\mapsto
0\in\mathcal{T}^4(F)$. Since $\frac{2}{5}q=8v_0(\delta)-\delta^2$,
we see that $\frac{2}{5}q\mapsto 8\bar{v}_0$. If $\lambda_0$ is odd,
then the image of $\mathrm{Hdg}^4(F)$ in $\mathcal{T}^4(F)$ is
generated by $8\bar{v}_0$, which is an element of order 5. To
conclude the second case, we only need to note the fact that for any
even $\lambda_0$, the image of $\frac{1}{8}(\lambda_0^2-\delta^2)$
in $\mathcal{T}^4(F)$ is an element of order 2.
\end{proof}

\begin{cor}
Let $(F,\lambda_0)$ be a very general polarized hyperk\"ahler
manifold of $K3^{[2]}$-type. \\
(i) The Hodge conjecture holds true for $F$.\\
(ii) Let $\mathcal{Z}^4(F)=\mathrm{Hdg}^4(F)/\mathrm{Alg}^4(F)$. If
$\lambda_0$ is odd, then there is a surjection
$\Z/3\Z\twoheadrightarrow\mathcal{Z}^4(F)$; if $\lambda_0$ is even,
then there is a surjection
$\Z/24\Z\twoheadrightarrow\mathcal{Z}^4(F)$.
\end{cor}
\begin{proof}
To prove the Hodge conjecture for $F$, we only need to do this in
degree 4. By the computation carried out in the proof of the
theorem, we only need to show that $q$ is algebraic. We will need
the following explicit formula of the Chern classes of $F$ in the special
case $F=S^{[2]}$,
$$
 c_2(F)= 24 v_0-3\delta^2.
$$
See \cite[Lemma 9.3]{fourier}. This implies that $\frac{2}{5}q=\frac{1}{3}c_2(F)$. When $\lambda_0$
is odd, then $\mathcal{Z}^4(F)$ is generated by the image of
$\frac{1}{3}c_2(F)$. When $\lambda_0$ is even, then
$\mathcal{Z}^4(F)$ is generated by the image of
$\frac{1}{24}c_2(F)+\frac{1}{8}\lambda_0^2$. This proves the
corollary.
\end{proof}

We consider a special family of polarized
hyperk\"ahler manifolds. Let $X\subset\PP^5_\C$ be a smooth cubic
fourfold. Let $F=F(X)$ be the variety of lines on $X$. It is known
that $F$ is a hyperk\"ahler manifold of $K3^{[2]}$-type; see \cite[Proposition 2]{bd}. We have a natural
inclusion $F\subset
G(2,6)$. Let $\mathscr{E}$ be the restriction of the rank 2 quotient
bundle on the Grassmannian of lines in $\PP^5$. Set $g_1=c_1(\mathscr{E})$ and
$g_2=c_2(\mathscr{E})$. We take $\lambda_0=g_1$ to be the natural polarization on $F$. Recall from Hassett
\cite{hassett} that a cubic fourfold $X$ is special if $\mathrm{Hdg}^4(X)$ has rank at least two. Beauville-Donagi
\cite[Proposition 4]{bd} proved that the Abel--Jacobi homomorphism $\Phi:\HH^4(X,\Z)\rightarrow \HH^2(F,\Z)$ is an
isomorphism of Hodge structures. Hence $X$ is special if and only if the Picard rank of $F$ is at least two.

\begin{proof}[Proof of Proposition \ref{prop jacobian implies special}]
We only need to verify Assumption \ref{technical assumption} for one cubic fourfold $X$. We take $X$ to be general Pfaffian cubic fourfold. Then $F=S^{[2]}$ for some $K3$ surface $S$ of degree $14$. Let $\delta\in\Pic(F)$ be the boundary divisor and $\fb\in\Pic(S)$ be the polarization with $(\fb\cdot\fb)_S=14$. Then under the natural orthogonal decomposition
$$
\HH^2(F,\Z) = \HH^2(S,\Z) \oplus \Z\delta,\qquad b(\delta,\delta)=-2,
$$
the polarization $\lambda_0$ can be written as $\lambda_0 = 2\fb -5\delta$; see the proof of \cite[Proposition 6]{bd}. Then Assumption \ref{technical assumption} is readily verified. Then by Corollary \ref{jacobian picard rank}, we have
$\rk \Pic(F)\geq 2$. Hence $X$ is special.
\end{proof}

\begin{thm}\label{integral hodge cubic fourfolds}
Let $X$ be a very general cubic fourfold. Then the following are
true.\\
(i) The variety of lines $F=F(X)$ is not of Jacobian type.\\
(ii) The integral Hodge classes $\mathrm{Hdg}^4(F)$ is freely
generated by $g_2$ and $\frac{1}{3}(g_1^2-g_2)$.\\
(iii) The integral Hodge conjecture holds true for $F$ in degree 4.
\end{thm}
\begin{proof}
(i) Follows from the above proposition since a very general $X$ is not special. To prove (ii), we need to write $q$ explicitly in terms of $g_1$ and
$g_2$. Let $\Phi:\HH^4(X,\Z)\to\HH^2(F,\Z)$ be the Abel-Jacobi
isomorphism. Then the restriction of $\Phi$ to the
transcendental classes gives an isomorphism
$$
\Phi_0:\HH^4(X,\Z)_{\mathrm{tr}}\rightarrow\HH^2(F,\Z)_{\mathrm{tr}}
$$
which satisfies $b(\Phi(\alpha),\Phi(\beta))=-\alpha\cdot\beta$ for
all $\alpha,\beta\in\HH^4(X,\Z)_{\mathrm{tr}}$; see \cite[Proposition 6]{bd}. One
relation that we need is
\begin{equation*}
g_2\cdot\fa\cdot\fa'=0,\quad\forall\fa,\fa'\in\HH^2(F,\Z)_{\mathrm{tr}}.
\end{equation*}
This can be seen from the geometry. Let $Y\subset X$ be a general
hyperplane section, which is a smooth cubic threefold. Then $g_2$ is
represented by the surface of lines on $Y$. For any transcendental
classes $\fa$ and $\fa'$, we can find the corresponding
transcendental classes $\alpha,\alpha'\in\HH^4(X,\Z)_{\mathrm{tr}}$
such that $\fa=\Phi(\alpha)$ and $\fa'=\Phi(\alpha')$. When $\alpha$
is transcendental, we always have $\deg(\alpha|_Y)=0$. By the formula of
the number of secant lines of a pair of curves on a cubic
hypersurface obtained in \cite[Lemma 3.10]{relations}, we have
$$
g_2\cdot\fa\cdot\fa'=5\deg(\alpha|_Y)\deg(\alpha'|_Y)=0.
$$
Similarly, we have $g_2\cdot g_1\cdot\fa=0$ and $g_2g_1^2 =45$. The
self-intersection of $g_1$ is given by
$$
g_1^4=3b(g_1,g_1)^2=108.
$$
Based on these identities, we get
$$
(\frac{1}{6}g_1^2-\frac{4}{15}g_2)\cdot\fa\cdot\fb=b(\fa,\fb),\quad\forall
\fa,\fb\in\HH^2(F,\Z).
$$
This implies that $\frac{1}{6}g_1^2-\frac{4}{15}g_2=\frac{1}{25}q$.
Hence we have
$$
\frac{1}{8}(\frac{2}{5}q+g_1^2)=\frac{1}{3}(g_1^2-g_2).
$$
Then (ii) follows from Theorem \ref{thm on hodge classes} and the fact that $g_1$ is even.

Note that $\frac{1}{3}(g_1^2-g_2)$ is represented by the surface of
all lines meeting a given line, see \cite[\S0]{voisin}. This proves (iii).
\end{proof}

\section{Rational cubic fourfolds}
In this section, we give some evidence of our Conjecture
\ref{rationality conjecture}. We first state some well-known
formulas for blow-ups. We then show that rationality of a fourfold
implies the existence of a Hodge theoretically special family of
rational curves parametrized by a surface. Using the techniques and
constructions of \cite{relations}, we were able to relate this to
the variety of lines and prove Theorem \ref{thm on rationality imply
potential}.
\subsection{Blow up formulas}
Let $Y$ be a smooth projective variety of dimension 4 and $Z\subset
Y$ a smooth closed subvariety. Let $\tilde{Y}=\mathrm{Bl}_Z(Y)$ be
the blow up of $Y$ along $Z$ and $E\subset\tilde{Y}$ the exceptional
divisor. Then we have the following commutative diagram
$$
\xymatrix{
 E\ar[r]^j\ar[d]_\pi &\tilde{Y}\ar[d]^\sigma\\
 Z\ar[r]^i & Y
}
$$
\begin{prop}\label{blow up formulas}
Let $Y$, $Z$ and $\tilde{Y}$ be as above. We use $\hat\oplus$ to
denote orthogonal decomposition. Then the following are
true.\\
(i) If $Z$ is a point, then
$$
\HH^4(\tilde{Y},\Z)\cong\HH^4(Y,\Z)\hat\oplus \Z[\PP^2],
$$
where the extra class $[\PP^2]$ is the class of a linear $\PP^2$ in
the exceptional divisor $E\cong\PP^3$. Furthermore we have
$[\PP^2]^2=-1$.\\
(ii) If $Z$ is a curve, then
$$
\HH^4(\tilde{Y},\Z)=\HH^4(Y,\Z)\hat\oplus (\Z j_*\xi\oplus\Z j_*f),
$$
where $\xi$ is the class of the relative $\calO(1)$-bundle of
$\pi:E=\PP(\mathscr{N}_{Z/Y})\to Z$ and $f$ is the class of a fiber
of $\pi$. The intersection matrix of $j_*\xi$ and $j_*f$ is
$$
\begin{pmatrix}
d &-1\\
-1 &0
\end{pmatrix},
$$
where $d=\deg(\mathscr{N}_{Z/Y})$.\\
(iii) If $Z$ is a surface, then we have a canonical identification
$$
\HH^4(\tilde{Y},\Z)=\HH^4(Y,\Z)\hat\oplus\HH^2(Z,\Z)(-1),
$$
where $\HH^2(Z,\Z)\to\HH^4(\tilde{Y},\Z)$ is given by $\fa\mapsto
-j_*\pi^*\fa$ and the intersection form is given by
$$
j_*\pi^*\fa\cdot j_*\pi^*\fa'=-(\fa\cdot\fa')_Z.
$$
The projection $\HH^4(\tilde{Y},\Z)\to\HH^2(Z,\Z)$ is given by
$\alpha\mapsto \pi_*j^*\alpha$.
\end{prop}
\begin{proof}[Sketch of Proof]
 These formulas are classical. For example (see \S2.1 of \cite{peters}), one can use the Mayer--Vietoris sequence to get the following short exact
sequence
$$
\xymatrix{
 0\ar[r] &\HH^4(Y,\Z) \ar[r]^{(i^*,\sigma^*)\quad\quad\,\,\,} &\HH^4(Z,\Z)\oplus \HH^4(\tilde{Y},\Z)
\ar[r]^{\,\,\quad\quad\pi^* - j^*} &\HH^4(E,\Z) \ar[r] &0.
}
$$
Since $E$ is a projective bundle over $Z$, we can write $\HH^4(E)$ in terms of the cohomology of $Z$. Then the formulas
(without the bilinear form) in the proposition follow from the above sequence. To get the bilinear form, it remains to
carry out a routine computation of intersection numbers using the Chern class of the relative $\calO(1)$-bundle of
$E\rightarrow Z$. Please also see the proof of Lemma \ref{splitting lemma}.
\end{proof}

\begin{cor}\label{blow up formula for transcendental}
If $Z$ is a point or a curve, then
$$
\HH^4(\tilde{Y},\Z)_{\mathrm{tr}}=\HH^4(Y,\Z)_{\mathrm{tr}}.
$$
If $Z$ is a surface, then
$$
\HH^4(\tilde{Y},\Z)_{\mathrm{tr}}=\HH^4(Y,\Z)_{\mathrm{tr}}\hat\oplus\HH^2(Z,\Z)_{\mathrm{tr}}(-1).
$$
\end{cor}

\subsection{Rational fourfolds}
Let $Y$ be a smooth projective fourfold that is rationally
connected. Let $S=\amalg S_i$ be a smooth (not necessarily
irreducible) surface. Let $K(S)=\oplus K(S_i)$ be the product of the
function fields of its components. Any nontrivial morphism
$$
\varphi_{K(S)}: \PP^1_{K(S)}=\coprod \PP^1_{K(S_i)}\rightarrow Y
$$
induces an Abel-Jacobi map
$$
\alpha=\pi _*\varphi^*:\HH^4(Y,\Z)_{\mathrm{tr}}\to
\HH^2(S,\Z)_{\mathrm{tr}}.
$$
This can be defined as follows. First, the morphism $\varphi$ can be
defined over a dense open subset $U\subset S$. Namely, we have the
following diagram
$$
\xymatrix{
  \PP^1\times U\ar[d]_{\pi_U}\ar[r]^{\quad\varphi_U} & Y\\
  U &
}
$$
This picture can be completed into a proper family of rational
curves on $Y$ as follows
$$
\xymatrix{
 \mathscr{C}\ar[r]^\varphi\ar[d]_\pi &Y\\
 \tilde{S} &
}
$$
This allows us to define the usual Abel-Jacobi map
$$
\pi_*\varphi^*:\HH^4(Y,\Z)_{\mathrm{tr}} \rightarrow
\HH^2(\tilde{S},\Z)_{\mathrm{tr}}.
$$
It is well-known that the transcendental lattice of a surface is a
birational invariant. Since $S$ is birational to $\tilde{S}$, we
have the canonical isomorphism
$$
\HH^2(\tilde{S},\Z)_\mathrm{tr}\cong\HH^2(S,\Z)_{\mathrm{tr}}.
$$
This gives the homomorphism
$$
\alpha: \HH^4(Y,\Z)_{\mathrm{tr}}\rightarrow
\HH^2(S,\Z)_\mathrm{tr}.
$$
One checks that this definition of $\alpha$ is independent of the
choice of the spreading $\varphi_U$ and the completion $\tilde{S}$.
Note that by construction, the morphism $\pi:\mathscr{C}\rightarrow
\tilde{S}$ has a rational section.
\begin{defn}
We say that $S$ \textit{receives the cohomology of $Y$ with index
$e$} via a rational curve $\varphi_{K(S)}: \PP^1_{K(S)}\to Y$ as
above if the associated Abel-Jacobi map satisfies
$$
(\alpha(x)\cdot\alpha(y))_S=- e\,(x\cdot y)_Y,\quad\forall
x,y\in\HH^4(Y,\Z)_{\mathrm{tr}}.
$$
When a polarization $H$ of $Y$ is fixed, the \textit{degree} of the
rational curve $\varphi_{K(S)}$ is defined as the degree of
$\varphi(\PP^1_{s})$ for a general closed point $s\in S(\C)$. We
will call $\varphi_{K(S)}$ a \textit{line} if its degree is 1.
\end{defn}

Let $Y$ be a smooth projective fourfold. Assume that $Y$ is
rational. Then there is a birational map $f:\PP^4\dashrightarrow Y$.
We resolve the indeterminacy of $f$ by successive blow-ups along
smooth centers and get the following picture
$$
\xymatrix{
  &Y'\ar[ld]_\sigma\ar[rd]^{\tilde{f}} &\\
  \PP^4 & &Y
}
$$
where $\sigma$ is the seccessive blow-up. Let $S_1,\ldots,S_r$ be
the smooth surfaces that appear as the center of the blow-up at some
step. Let $T_1,\ldots,T_r\subset\PP^4$ be the images of the $S_i$'s.
\begin{defn}\label{defn of simple blow up}
We say that successive blow-up $\sigma$ is \textit{simple} if the
$T_i$'s are all surfaces distinct from each other.
\end{defn}
We assume that $\sigma$ is simple. This implies that $S_i\to T_i$ is birational and
$\sigma^{-1}(t)\cong\PP^1$ for general $t\in T_i$. Let
$$
S=\coprod_{i=1}^{r} S_i,\qquad T=\coprod_{i=1}^{r}T_i.
$$
Fix a general point $x\in Y$ with its pre-image in $\PP^4$ being $x'$.
For a general point (meaning from a dense open subset) $t\in T$, we
use $L_t\subset \PP^4$ to denote the line passing through $x'$ and
$t$. Let $L'_t\subset Y'$ be the strict transform of $L_t$ and
$C_t\subset Y$ be its image in $Y$. Note that $x\in C_t$ for all
such general $t$. As $t$ varies in a dense open subset of $S$, this
gives a rational curve
$$
\varphi_{K(S)}:\PP^1_{K(S)}\rightarrow Y
$$
that passes through $x\in X$. For a general point $t\in T$, we also
define $E'_t=\sigma^{-1}(t)\subset Y'$ and $E_t\subset Y$ be its
image in $Y$. Since $S\to T$ is birational, we see that $E_t$, as
$t$ runs through general points of $T$, defines a rational curve
$$
\tilde{\varphi}_{K(S)}:\PP^1_{K(S)}\to Y.
$$
\begin{lem}\label{lemma on rational fourfolds}
Let notations and assumptions be as above, then $S$ receives the
cohomology of $Y$ via $\varphi_{K(S)}$ (resp.
$\tilde\varphi_{K(S)}$) with index $-1$ (resp. $1$).
\end{lem}
\begin{proof}
By (iii) of Proposition \ref{blow up formulas} and Corollary
\ref{blow up formula for transcendental}, we see that
$$
\tilde{\pi}_*\tilde{\varphi}^*:\HH^4(Y,\Z)_{\mathrm{tr}}\to\HH^2(S,\Z)_{\mathrm{tr}}=\HH^4(Y',\Z)_{\mathrm{tr}}
$$
is equal to $\tilde{f}^*$. This proves the case of
$\tilde{\varphi}_{K(S)}$. Let $\{\tilde{L}_b:b\in B=\PP^3\}$ be all
the lines on $\PP^4$ passing through $x'$. Then for general $b\in
B$, the line $\tilde{L}_b$ gives rise to a rational curve
$C'_b\subset Y$. When $\tilde{L}_b$ specializes to a line $L_t$ for
some general point $t\in T$, the rational curve $C'_b$ specializes
to the nodal curve $C_t\cup E_t$. This gives us the following
picture
$$
\xymatrix{
 \mathscr{C}\ar[r]^{g'}\ar[d]^p &\mathscr{C}'\ar[d]^q\ar[r]^h &
 Y\\
 S\ar[r]^g &B &
}
$$
Here $\mathscr{C}$ is the total space of $\{E_t\cup C_t\}_{t\in S}$
and $\mathscr{C}'$ is the total space of $\{\tilde{L}_b\}_{b\in B}$.
Then by construction, we have
$$
\pi_*\varphi^*+\tilde\pi_*\tilde\varphi^*=p_*(g')^*h^*=g^*q_*h^*
$$
as homomorphisms
$\HH^4(Y,\Z)_{\mathrm{tr}}\to\HH^2(S,\Z)_{\mathrm{tr}}$. Since
$\HH^2(B,\Z)_{\mathrm{tr}}=0$, we get $g^*=0$ and hence
$$
\pi_*\varphi^*=-\tilde\pi_*\tilde\varphi^*.
$$
This proves the case of $\varphi_{K(S)}$.
\end{proof}

\subsection{Rational cubic fourfolds}
Let $X$ be a smooth cubic fourfold and $F=F(X)$ its variety of lines.
\begin{prop}\label{prop jacobian type and index}
The hyperk\"ahler manifold $F$ is of Jacobian type if and only if
there exists a surface $S$ that receives the cohomology of $X$ with
index $1$ via lines.
\end{prop}
\begin{proof}
By \cite[Proposition 4]{bd}, the universal line defines the Abel-Jacobi isomorphism
$$
\Phi:\HH^4(X,\Z)\rightarrow\HH^2(F,\Z).
$$
If there is a line $\varphi_{K(S)}:\PP^1_{K(S)}\to X$, then this
gives a rational map $f:S\dasharrow F$. By replacing $S$ by a
birational model, we may assume that $f$ is a morphism. Then we have
$$
f^*\circ\Phi=\pi_*\varphi^*:\HH^4(X,\Z)_{\mathrm{tr}}\rightarrow\HH^2(S,\Z)_{\mathrm{tr}}.
$$
As was shown in \cite[Proposition 6]{bd},
$$
b(\Phi(x),\Phi(y))=-(x\cdot y),\quad x,y\in\HH^4(X,\Z)_{\mathrm{tr}}.
$$
We see that $S$ represents a minimal class if and only if it
receives cohomology of $X$ with index $1$ via $\varphi_{K(S)}$.
\end{proof}

\begin{rmk}\label{rmk known cases}
We have two well-known examples of rational cubic fourfolds. The first example is the Pfaffian cubic fourfolds and the second example is the cubic fourfolds containing a plane where a certain Brauer element vanishes. We will show that Conjecture \ref{rationality conjecture} holds in each example.

 Let $X$ be a general Pfaffian cubic fourfolds. It was proved by Beauville--Donagi \cite{bd} that $X$ is rational and the corresponding variety $F$ of lines is isomorphic to $S^{[2]}$ for some $K3$ surface $S$. Thus our Conjecture \ref{rationality conjecture} holds for such $X$.

Let $X$ be a general cubic fourfold containing a plane $\Pi\cong\PP^2\subset X$. Let $\tilde{X}$ be the blow-up of $X$ along $\Pi$. The projection from the plane $\Pi$ determines a morphism $\pi: \tilde{X}\rightarrow \Pi'= \PP^2$. A general fiber of $\pi$ is a smooth quadric surface and hence isomorphic to $\PP^1\times \PP^1$. There is a smooth curve $\Delta\subset\Pi'$ of degree $6$ such that for all $t\in \Delta$ the fiber $\pi^{-1}t$ is isomorphic to the cone over a conic. Note that a general fiber of $\pi$ has two rulings and by singling out one of them, we get a double cover $r:S\rightarrow \Pi'$ which ramifies along $\Delta$. Hence $S$ is a $K3$ surface of degree 2. Consider the actual lines in a given ruling, we get a divisor $D\subset F$ which has a natural morphism $s:D\rightarrow S$ whose closed fibers are all isomorphic to $\PP^1$. Associated to this situation, we have an element $\alpha\in\mathrm{Br}(S)$ that corresponds to the fibration $D\rightarrow S$. Now we assume that $\alpha = 0$ and consequently $s$ admits a rational section $\lambda : S \dashrightarrow D$. In this case, $X$ is rational since the vanishing of $\alpha$ implies that $\pi$ has a rational section and hence $\tilde{X}$ is rational; see \cite{hassett2}. By blowing up $S$, we get a morphism $\tilde{\lambda}: \tilde{S}\rightarrow D$. Together with the natural inclusion $D\subset F$, we get a surface $\phi:\tilde{S}\rightarrow F$. We next show that $\tilde{S}$ represents a minimal class on $F$ and thus Conjecture \ref{rationality conjecture} is verified in this case. In $\HH^4(X,\Z)$, we have two natural classes, namely $\Pi$ and $h^2$. Let $L$ be the orthogonal complement of $\langle h^2,\Pi \rangle$ in $\HH^4(X,\Z)$. Voisin \cite[Proposition 2]{voisin} proved that $\Phi(L)|_D \subset s^*\HH^(S,\Z)^0$ where $\HH^2(S,\Z)^0$ is the orthogonal complement of $r^*c_1(\calO_{\Pi'}(1))$ in $\HH^2(S,\Z)$. Furthermore, Voisin also proved that if $x,y\in L$ and if $x',y'\in\HH^2(S,\Z)^0$ such that $\Phi(x)|_D = s^*x'$, $\Phi(y)|_D = s^*y'$, then $(x\cdot y)_X = -(x'\cdot y')_S$. By \cite[Proposition 6]{bd}, we have $b(\Phi(x),\Phi(y)) = -(x\cdot y)_X$. It follows that
$$
(\phi^*u\cdot\phi^*v)_{\tilde{S}} = (\tilde{\lambda}^* s^*x'\cdot \tilde{\lambda}^* s^* y')_{\tilde{S}} = (x'\cdot y')_S = b(u,v),
$$
for all $u=\Phi(x)$ and $v=\Phi(y)$ in $\Phi(L)$, where $x,y\in L$. Note that $\HH^2(F,\Z)_{\mathrm{tr}}\subset \Phi(L)$. Hence the above equality implies that $\phi_*[\tilde{S}]$ is minimal class.
 \qed
\end{rmk}

\begin{prop}
Let $C\cong\PP^1\subset X$ be a general rational curve of degree $e$
on $X$ and $S_C$ the surface of lines meeting $C$. Then the lines
defined by $S_C$ gives a morphism $\varphi_{K(S)}:\PP^1_{K(S)}\to X$
and $S_C$ receives the cohomology of $X$ with index $2e$.
\end{prop}
\begin{proof}
This follows from the Prym-Tjurin construction obtained in \cite{pt}.
\end{proof}

Given the above result, we would like to ask the following
\begin{ques}\label{question}
For a very general cubic fourfold $X$, is there a surface $S$ that
receives the cohomology of $X$ via a rational curve with an odd
index?
\end{ques}
We expect that the answer to this question is negative and such a
negative answer should imply nonrationality of $X$. To see how this
question is related to Conjecture \ref{rationality conjecture}, we
recall some constructions and results of \cite{relations}. Let
$f:C=\PP^1\to X$ be a smooth rational curve of degree $e$ on $X$. A
\textit{secant line} of $C$ is a line on $X$ that meets $C$ in 2
points. We say that $C$ is \textit{well-positioned} if there are
only finitely many distinct secant lines of $C$. Assume that $C$ is
well-positioned. A general pair of points $(x,y)\in\Sym^2(C)=\PP^2$
determines a line $L_{x,y}\subset\PP^5$. The line $L_{x,y}$
intersects $X$ in a third point $z$ unless $L_{x,y}$ is a secant
line of $C$. As $(x,y)$ varies, we get a rational map
$\phi_0:\Sym^2(C)\dasharrow X$, $(x,y)\mapsto z$. Let $\Sigma$ be
the blow up of $\Sym^2(C)$ at the points corresponding to the secant
lines, then $\phi_0$ extends to a morphism $\phi:\Sigma\to X$. This
rational surface $\Sigma$ is called the \textit{residue surface} of
$C$. The condition $z=x$ or $y$ defines a curve
$\tilde{C}\subset\Sigma$. The restriction of $\phi$ to $\tilde{C}$ gives a morphism $\tilde{C}\to
C$ which has degree $e-2$. There are several natural divisors on $\Sigma$. Let
$x\in C$ be a general closed point, we define $D_x\subset \Sigma$ be
the locus $\{(x,x'):x'\in C\}\subset\Sym^2(C)$. Let $\fa\subset C$ be
a hyperplane section and we write $\fa=\sum_{i=1}^{e} x_i$. We
define
$$
D_\fa =\sum_{i=1}^{e} D_{x_i}.
$$
Let $\xi=\phi^*h$ be the class of a hyperplane. The locus $(x,x)$ for
$x\in C$ defines $\Delta\subset \Sigma$. The class of $\Delta$ is
divisible by 2 in the Picard group of $\Sigma$. Let $\Delta_0=
-\frac{1}{2}\Delta$. Then the following holds true in
$\Pic(\Sigma)$,
\begin{equation}\label{key relation}
\tilde{C}-D_{\fa}-\Delta_0-\xi +\sum_{i=1}^{N}E_i=0,
\end{equation}
where $E_i$, $i=1,\ldots, N$, are all the exceptional divisors of
the blow up $\Sigma\to\Sym^2(C)$ and $\fa=h|_{C}$. This is proved in Proposition 3.6
of \cite{relations}. Viewed as algebraic cycles on $X$ modulo rational
equivalence, we have
$$
\phi_*\tilde{C}=(e-2)C,\quad\phi_*D_\fa=-eC+a_1h^3,\quad\phi_*\Delta_0=C+a_2h^3, \quad\phi_*\xi=a_3h^3,
$$
for some $a_1,a_2,a_3\in\Z$; see the proof of \cite[Theorem 4.2]{relations}.

If $C$ varies in a family, then the above
constructions can be carried out in families. Furthermore, the
construction does not require the base field to be algebraically
closed. Now assume that $S$ is a surface and
$$
\xymatrix{
 \mathscr{C}\ar[r]^\varphi\ar[d]^\pi &X\\
 S &
}
$$
be a family of rational curves of degree $e$ on $X$. Assume that for
a general point $s\in S$, the curve $\mathscr{C}_s$ is
well-positioned. By carrying out the above construction in this
family, we see that there is a dense open subset $U\subset S$ and
$\tilde{\Sigma}\to U$, $\tilde{\mathscr{C}}\subset \tilde{\Sigma}$,
$\Xi\subset\tilde{\Sigma}$, $\tilde{D}\subset\tilde{\Sigma}$,
$\tilde{\Delta}_0\subset \tilde{\Sigma}$ and $\tilde{E}\subset
\tilde{\Sigma}$ such that for any point $s\in U$, the fiber
$(\tilde{\Sigma})_s$ is the residue surface of $\mathscr{C}_s$, and
$\tilde{\mathscr{C}}$, $\Xi$, $\tilde{D}$, $\tilde{\Delta}_0$ and
$\tilde{E}$ correspond to $\tilde{C}$, $\xi$, $D_\fa$, $\Delta_0$
and $\sum E_i$ respectively. We view
$\tilde{\mathscr{C}}$, $\Xi$, $\tilde{D}$, $\tilde{\Delta}_0$ and
$\tilde{E}$ as elements in $\CH_3(\tilde{\Sigma})$. Then we define
$$
\tilde{\Gamma}=\tilde{\mathscr{C}}-\tilde{D}-\tilde{\Delta}_0-\Xi+\tilde{E}\in\CH_3(\tilde{\Sigma}).
$$
Note that there is a natural morphism $\tilde{\phi}:\tilde{\Sigma}\rightarrow U\times X$. We define $\Gamma\in\CH_3(S\times X)$ to be the closure of $\tilde{\phi}_*\tilde{\Gamma}$. By abuse of notation, we will simply write $\Gamma = \tilde{\phi}_*\tilde{\Gamma}$. Let $\eta_S$ be the generic point of $S$. Note that
$\Gamma|_{\eta_S\times X}$ defines an element in $\CH_1(\eta_S\times
X)$. By doing the above construction of residue surface over the
point $\eta_S$, then the relation \eqref{key relation} says
$$
\Gamma|_{\eta_S\times X}=0.
$$
Note that we have
$$
\CH_1(\eta_S\times X)=\CH^3(\eta_S\times X)=\varinjlim\CH^3(U\times X)
$$
where $U$ runs through open dense subsets of $S$. Hence we see that
$\Gamma|_{U\times X}=0$ for some open dense $U\subset S$. Let
$B=S-U$ and $\tilde{B}$ be a desingularization of $B$. By shrinking
$U$, we may assume that $B$ is a curve. Then we see that
$$
\Gamma\in \mathrm{Im}(\CH_3(\tilde{B}\times X)\to\CH_3(S\times X)).
$$
It follows that the homomorphism $[\Gamma]^*:\HH^4(X,\Z)_\mathrm{tr}
\to\HH^2(S,\Z)_\mathrm{tr}$ factors through $\HH^0(\tilde{B},\Z)$.
This implies that $[\Gamma]^*=0$. Note that
$\tilde{\phi}_*\tilde{\mathscr{C}}=(e-2)\mathscr{C}\in\CH_3(S\times X)$, we get
$$
[\tilde{\phi}_*\tilde{\mathscr{C}}]^*=(e-2)[\mathscr{C}]^*.
$$
Consider
$\Gamma_1=\tilde{\phi}_*(\tilde{D}+e\tilde{C}-\frac{a_1}{a_3}\Xi)\in\CH_3(S\times
X)_\Q$ and run the same argument as for $\Gamma$, we get
$$
[\tilde{\phi}_*\tilde{D}]^*=-e[\mathscr{C}]^*+\frac{a_1}{a_3}[\tilde{\phi}_*\Xi]^*.
$$
Similarly, we have
$$
[\tilde{\phi}_*\tilde{\Delta}_0]^*=[\mathscr{C}]^*+\frac{a_2}{a_3}[\tilde{\phi}_*\Xi]^*.
$$
Note that $\Xi$ is simply the pull back of a hyperplane $H\subset X$
via the natural morphism $\tilde{\Sigma}\to X$. It follows that
$[\tilde{\phi}_*\Xi]^*:\HH^4(X,\Z)_{\mathrm{tr}}\to \HH^2(S,\Z)_\mathrm{tr}$
factors through $\HH^4(H,\Z)$. Since $\HH^4(H,\Z)\cong\Z$ only has
algebraic classes, we conclude that
$$
[\tilde{\phi}_*\Xi]^*=0.
$$
Put all these identities together, we get
\begin{equation}\label{key equation 1}
[\tilde{E}]^*=(3-2e)[\mathscr{C}]^*,\quad\text{on }\HH^4(X,\Z)_\mathrm{tr}.
\end{equation}
The generic fiber of $\tilde{E}\to S$ consists of finitely many
lines, namely the secant lines of the corresponding curve
$\mathscr{C}_s$. By selecting one of the secant lines, we get a
factorization $\tilde{E}\to\Gamma_S\to S$ and a morphism
$\Gamma_S\to F$. The surface $\Gamma_S$ gives rise to a
correspondence from $S$ to $F$. Then the equation \eqref{key
equation 1} can be rewritten as
\begin{equation}\label{key equation 2}
[\Gamma_S]^*\circ\alpha=(3-2e)[\mathscr{C}]^*,
\end{equation}
where $\alpha:\HH^4(X,\Z)_{\mathrm{tr}}\to\HH^2(F,\Z)_{\mathrm{tr}}$
is the Abel-Jacobi isomorphism. Now we can give a partial relation
between Question \ref{question} and Conjecture \ref{rationality
conjecture}.

\begin{prop}\label{proposition on odd index}
Let $X$ be a cubic fourfold and $F$ its variety of lines. Assume
that there exist a surface $S$ receiving the cohomology of $X$ with
odd index $e_0$ via a well-positioned rational curve. Then $F$ is
potentially of Jacobian type.
\end{prop}
\begin{proof}
Let $\mathscr{C}\to S$ be the rational curve defined by $S$. By
assumption, we have
$$
[\mathscr{C}]^*x\cdot [\mathscr{C}]^*y =-e_0(x\cdot y), \quad
x,y\in\HH^4(X,\Z)_{\mathrm{tr}}.
$$
Let $\Gamma_S$ be the surface parameterizing the secant lines as
constructed above. Then the equation \eqref{key equation 2} gives
$$
[\Gamma_S]^*u\cdot [\Gamma_S]^*v=e_0(2e-3)b(u,v),\quad
u,v\in\HH^2(F,\Z)_{\mathrm{tr}}.
$$
Note that $e_0(2e-3)$ is again odd. Take
$a=\frac{e_0(2e-3)-1}{2}\in\Z$. Let $l\subset X$ be a general line
and $f:S_l\to F$ the surface of lines meeting $l$. It is known that
$$
f^*u\cdot f^*v=2b(u,v), \quad u,v\in\HH^2(F,\Z)_{\mathrm{tr}}.
$$
Then we can take $S_1=S\coprod S_l$ and
$\Gamma_1=\Gamma_S-a\Gamma_f$, where $\Gamma_f$ is the graph of $f$.
Then we have
$$
[\Gamma_1]^*u\cdot [\Gamma_1]^*v=b(u,v),\quad u,v\in\HH^2(F,\Z)_\mathrm{tr}.
$$
This means that $F$ is potentially of Jacobian type.
\end{proof}
\begin{proof}[Proof of Theorem \ref{thm on rationality imply potential}]
This is essentially a combination of Lemma \ref{lemma on rational
fourfolds} and Proposition \ref{proposition on odd index}. The only
thing that we need to note is that in the proof of Lemma \ref{lemma
on rational fourfolds}, we can choose the point $x'$ general enough
such that $C_t$ is well-positioned for general $t$.
\end{proof}

\end{document}